\documentclass[11pt]{amsart}
\usepackage{color}
\usepackage{amsfonts, amsmath, amssymb}
\usepackage{wasysym}
\usepackage{graphicx}

\definecolor{RedClr}{rgb}{1,0,0}
\definecolor{BlueClr}{rgb}{0,0,1}
\definecolor{TextColor}{rgb}{0,0,0.5}
\definecolor{Violet}{rgb}{0.5,0,1}
\definecolor{Bordeaux}{rgb}{1,0.3,0.4}

\newtheorem{thm}{Theorem}[section]
\newtheorem{cor}[thm]{Corollary}
\newtheorem{lem}[thm]{Lemma}
\newtheorem{prop}[thm]{Proposition}
\theoremstyle{definition}
\newtheorem{defn}[thm]{Definition}
\theoremstyle{remark}
\newtheorem{rem}[thm]{Remark}
\numberwithin{equation}{section} \theoremstyle{quest}

\numberwithin{equation}{section} \theoremstyle{prob}

\numberwithin{equation}{section} \theoremstyle{answer}

\numberwithin{equation}{section}
\numberwithin{equation}{section}
\newtheorem{exmps}[thm]{Examples}

\begin{document}

\title[Geometric Sampling of Networks]{Geometric Sampling of Networks}

\author{Vladislav Barkanass, J\"urgen Jost and Emil Saucan$^*$ }
\address{ORT Braude College of Technology, Karmiel 2161002, Israel, 
	\and J\"{u}rgen Jost, Max Planck Institute for Mathematics in the Sciences, Leipzig 04103, Germany,
	\and Emil Saucan, ORT Braude College of Technology, Karmiel 2161002, Israel}  
	\email{vladigr1@gmail.com,  jost@mis.mpg.de, semil@braude.ac.il \newline $^*$ Corresponding author: semil@braude.ac.il}

\thanks{J. Jost and E. Saucan were partly supported by the German-Israeli Foundation Grant 
	I-1514-304.6/2019. }

\maketitle
%


\begin{abstract}
Motivated by the  methods and results of manifold sampling based on Ricci curvature, we propose a similar approach for networks. To this end we make appeal to three types of discrete curvature, namely the graph Forman-, full Forman- and Haantjes-Ricci curvatures for edge-based and node-based sampling. We present the results of experiments on real life networks, as well as for square grids arising in Image Processing. Moreover, we consider fitting Ricci flows and we employ them for the detection of networks' backbone. 
We also develop embedding kernels related to the Forman-Ricci curvatures and employ them for the detection of the coarse structure of networks, as well as for network visualization with applications to SVM. The relation between the Ricci curvature of the original manifold and that of a Ricci curvature driven discretization is also studied. 
\end{abstract}


\section{Introduction}

People study and make appeal to the coarse geometry of networks, constantly and for a long time now, whether they do it conscientiously (rarely) or not (in most cases). We shall demonstrate this shortly, however, to make this even clear, let us specify what we, informally, mean by ``Coarse Geometry": The study of the geometric (topological) properties, without ``looking at'' the small scales. In other words, one does not discern between objects that, viewed from sufficiently far away, look the same. Perhaps the simplest and immediate example of such a large scale geometry (or, in John Lott's suggestive words, ``Mr. Magoo geometry'') is given by the integer grid in the Euclidean plane. 
(For more insights and technical definitions and results, see, e.g \cite{Gr-carte}, \cite{Roe}, \cite{Kanai}, as well as Appendix 1.)

The Complex Networks community has been exposed to this approach -- and makes by now use of it -- via the notion of {\it Gromov hyperbolicity}, itself stemming from Gromov's seminal work on {\it hyperbolic groups} \cite{Gr-HypGrps}. 
The -- perhaps somewhat theoretical in the eyes of many practitioners -- notion above is connected to the notion of {\it network backbone}, underlying the long-distance relations between major network regions. Indeed, the connection between negative curvature and these ``communication highways" in networks has been emphasized in \cite{Ni2015}. We recall below its formal definition, as introduced in \cite{WSJ2}:

\begin{defn}
	\label{def:backbone}
	We denote the \emph{backbone} of a network $G=\lbrace V, E \rbrace$ as a subnetwork $G' = \lbrace V', E' \rbrace$ ($V' \subseteq V$, $E' \subseteq E$) that captures structurally important nodes (\emph{hubs}) and edges (\emph{bridges}). A node is typically termed \emph{hub} if it has a high degree and a high betweenness centrality. \emph{Bridges} denote edges that govern the mesoscale structure of $G$, for instance by forming long-range connections between communities. The backbone $G'$ is \emph{structure-preserving}, i.e., its structural features (e.g., node degree distribution, community structure) are representative of $G$.
\end{defn}

The intuition behind this definition is quite simple: Since high-degree nodes (hubs) form the centers of the major network communities, the edges that form strong connections between them bridge the corresponding communities. The motivation for the study of the network backbone stemmed from Communication Networks, more precisely from that of the Internet, where it is a quite technical and specific term: The edges belonging to the backbone represent the Internet's  ``highways'', and are therefore physically implemented by the fastest existing optical connections. Moreover, they are owned and operated by the so called {\it network backbone providers (NBPs)}, who are mainly governments and the largest  telecommunication companies (see, e.g., \cite{New}). However, the rational behind the notion of backbone is not restricted solely to the Internet, but it also arises naturally in Social Networks \cite{BK}, and beyond \cite{BGL}.

Before passing further on, let us mention here that the network backbone is also referred to the {\it core of the network} \cite{New}. It is precisely this more geometric term which we have in mind in our approach here in, where networks are viewed as 1-dimensional versions of manifolds, endowed with a coarse geometry.

Another, older, far better known, and permeating almost (if not) all of large scale applications of Computer Science, inadvertent application of coarse geometry is {\it sampling}. Indeed, by choosing certain points in a space, and discarding the ``less important'' points in their neighborhood, in conjunction (in many instances), with the related process of {\it clustering}, one in effect discards the small scale structure, as encoded in the discarded vicinities, and rather concentrates on the large scale, i.e. coarse structure. Clearly sampling the ``heavy'' curvature nodes or edges in a network, and using them as clustering centers is natural. In fact, such clustering based on the combinatorial version of Gaussian curvature, that is the {\it clustering coefficient} has been employed for clustering in networks for quite some time now \cite{Wa}, and it was only natural to extend this method to a proper notion of metric curvature that allows of the incorporation of weights, be they node or edge weights, accounting for the simultaneous presence of the two types of weights as well \cite{SA}. 
Thus we connect the two interpretations of the notion of coarse geometry, namely the {\it global} one, motivated largely by Gromov's hyperbolicity, and the {\it local} one, (as opposed to the infinitesimal one), obtained by the approximation of manifolds by networks.
We shall show, moreover, that sampling by {\it Ricci curvature}, is not just an heuristic, empirical procedure, and it has, in fact, deep mathematical justifications that also point to its full potential. We shall prove our assertion in Section 5 below.

The reminder of this paper is structured as follows: In the next section we bring the mathematical background, motivating our thrust for a Ricci curvature driven sampling of networks. Section 3 is dedicated to the introduction of three types of discrete  Ricci curvature, namely graph Forman, full Forman and Haantjes Ricci curvatures, that we deem particularly suited for the networks' sampling task as well to first experiments, on real-life networks as well as square grids that arise in Imaging. Furthermore, fitting Ricci flows are considered and applied to the networks' backbone detection. By ``fitting'', we mean Ricci flows in which the role of the Ricci curvature is taken by any of the consider discretization. 
We follow in Section 4 with an application of the Forman Ricci curvatures and associated Bochner Laplacians to the study of the coarse geometry of networks, via the development of fitting embedding kernels. An application to the visualization of kernel spaces and networks is also presented. In Section5we return to the motivating manifold sampling and show the connections between the Ricci curvature of the given manifold and the Forman curvatures of the resulting discretizations. The last section comprises a terse summary of the paper and an outlook towards the directions of further study that we deem more important. 


\section{Background: Ricci Curvature Based Sampling of Manifolds}\label{section: background}

Sampling by curvature, namely choosing the sampling points whose metric density is inverse proportional to curvature has been rediscovered and applied in Imaging and Graphics, as well as in more theoretical problems a number of times. (The relevant bibliography is far too extensive to include in the present paper, therefore we rather point the reader to the articles mentioned above and to the sources mentioned therein.) The most general method is perhaps the one in \cite{SAZ}, that also extends it to the sampling of more general signals. 

However, all the approaches mentioned above are based on {\it extrinsic} curvature, thus in practice necessitating first finding an isometric embedding in $\mathbb{R}^n$, a problem that is highly nontrivial for abstract manifolds.\footnote{This being in contrast with the Imaging/Graphics case where images/meshes are already embedded in $\mathbb{R}^3$.} Therefore, it is highly desirable to find a sampling method based on {\it intrinsic} curvature. Fortunately, it turns out that the basis for such a method exists for a long time in Geometry \cite{GP}, and it is based on {\it Ricci curvature} (see, e.g. \cite{Jo}), which is an intrinsic quantity. 

Before proceeding further, let us emphasize that the choice of Ricci curvature, instead, for instance, of {\it scalar curvature} \cite{Jo}, is not just a whim or fad, and stems from the absorption of the fundamental fact that networks are determined not by their members (nodes),  but rather by their connections (edges), and Ricci curvature is, on networks, a quantity attached to edges (as discretizations of vectors).

The basic idea resides in the constriction of so called {\it efficient packings}:

\begin{defn} \label{def:epsilon-nets}
	Let $\in M^n$ be a Riemannian manifolds and let $p_1,\ldots,p_{n_0}$ be points  $\in M^n$,  satisfying  the
	following conditions:
	\begin{enumerate}
		\item The set $\{p_1,\ldots,p_{n_0}\}$ is an $\varepsilon$-net on $M^n$, i.e. the
		balls $\beta^n(p_k,\varepsilon)$, $k=1,\ldots,n_0$ cover $M^n$;
		\item The balls (in the intrinsic metric of $M^n$) $\beta^n(p_k,\varepsilon/2)$ are pairwise
		disjoint.
	\end{enumerate}
	Then the set $\mathcal{N} = \{p_1,\ldots,p_{n_0}\}$ is called a {\it minimal
		$\varepsilon$-net} and the packing with the balls
	$\beta^n(p_k,\varepsilon/2)$, $k=1,\ldots,n_0$, is called an {\it
		efficient packing}. The set $\{(k,l)\,|\,k,l = 1,\ldots,n_0\; {\rm
		and}\; \beta^n(p_k,\varepsilon) \cap \beta^n(p_l,\varepsilon) \neq
	\emptyset\}$ is called the {\it intersection pattern} of the minimal
	$\varepsilon$-net  of the efficient packing. 
\end{defn}

There exists a canonical  simplicial complex
having as vertices the centers of the balls
$\beta^n(p_k,\varepsilon)$, which is constructed by adding a $k$-simplex for every collection of $k+1$ balls with nonempty intersection.

Following Kanai \cite{Kanai}, we call the graph $G(\mathcal{N})$ given by the 1-skeleton of the simplicial complex constructed above a {\it discretization} of $X$, with {\it separation} $\varepsilon$ and {\it covering radius $\varepsilon$} (or a {\it $\varepsilon$-separated net}). Further more, we say that $G(\mathcal{N})$ has {\it bounded geometry} iff there exists $\rho_0 > 0$, such that $\rho(p) \leq \rho_0$, for any vertex $p \in \mathcal{N}$, where $\rho(p)$ denotes the degree of $p$ (i.e. the number of neighbours of $p$). Furthermore, in the case of unbounded manifolds, we add, to the conditions in Definition \ref{def:epsilon-nets} above, the following requirement:
\vspace*{0.2cm}

(3) The graph $\mathcal{N}$ is maximal with respect to inclusion.
\vspace*{0.2cm}

If $M^n$ is  a closed, connected n-dimensional
Riemannian manifold such that it has sectional curvature $k_M$ bounded from
below by $k$, ${\rm diam}M^n$ bounded from above by $D$, and ${\rm Vol}M^n$
is bounded from below by $v$, such $\varepsilon$-nets can be  constructed by taking any maximal set of points with disjoint $\varepsilon/2$-balls, and 
the geometric generation process of such a maximal set is based on the close connection between the growth rate of volumes of balls in manifolds and Ricci curvature. Moreover, since only volumes of balls arguments are employed, one
can replace the last condition by the more general one ${\rm Ric}_M \geq (n-1)k$ (see,
e.g., [32]). 
%
%
Furthermore, the connection between volume and Ricci curvature can be extended to more general measures and to 
 the {\it generalized Ricci curvature} developed by Lott-Villani \cite{LV} and Sturm \cite{St}. Therefore, it is natural to seek and generalize the results of Grove and Petersen \cite{GP} to metric measure spaces.  Such an extension of the classical case construction to the metric measure spaces context does, indeed exist \cite{Sa11-1}. 

%


Most importantly,  the discretizations rendered in the process of proof, are, indeed, {\it coarsely equivalent} ({\it coarsely isometric}) -- see Definition \ref{eq:CoarseEmbed} -- to the original metric measure space, in a manner that is quite deep (see \cite{Sa11-1}, Theorem 5.6 and Theorem 5.11, as well as their corollaries). Still, the most important consequence, from our point of view, appears already in \cite{Sa11-1}, Corollary 4.6, 
namely the fact that the graphs (networks) obtained encode the essential topology (homotopy) of the sampled space. Moreover, it is a byproduct of the construction (see \cite{GP}, \cite{Sa11-1}) that even if the  graph-based reconstruction of the given space is only a coarse one, the number of such ``guesses'' is finite, and is, moreover, independent of the specific geometry of the manifold and the measure, and it depends only on bounds on dimension, volume, curvature and diameter.

It is therefore, only natural and intuitive to expect that a similar Ricci curvature based sampling should apply to weighted networks, viewed as metric measure spaces (with measures concentrated at the nodes and metric prescribed by the edge weights), should encode the essential topology of the network.


\begin{rem}
However, while most researchers have come to view graphs/networks as metric spaces, the model of {\it metric measure spaces} for weighted networks, is yet to be widely adopted by the Complex Networks community, even though it is a most natural way of describing the properties of such objects. To this end, a natural idea is to incorporate the node weights and edge weights into one expressive metric, thus rendering any weighted network into a ``honest to God'' metric space, whose geometric properties (curvature, geodesics, embeddings, etc.) can than be investigated with (more-or-less) classical tools. (An example of such a comprehensive metric is the so called {\it degree path metric} -- see, e.g. \cite{Ke}. Another well known such metric is the {\it resistance metric} (see, e.g. \cite{DD}). For the convenience of the reader we have expanded on these important metrics in Appendix 3.) 

It is worth noting that, initially, only edge weights were considered, and for such weights a global metric is standard, namely the so called {\it path metric} (which Computer Science students commonly confuse with the most popular algorithm for its computations, namely the {Dijkstra algorithm}.) 
A number of modeling problems, in particular some of Mathematical Biology motivated lately the search for ``good'' metrics that take into account node weights exclusively. (For such an example see, for instance, \cite{SA}.) 
From the applications viewpoint it become clear quite early that a comprehensive model of networks should include both edge and node weights (and that the combinatorial model is, a fortiori, quite unsatisfactory). The reason that such complete sets of weights were not previously employed in large scale experimental studies was due to the lack of a large enough sample of available data sets. 
However, as the field of Complex Networks evolved and expanded, better data sets have been published, thus the focus of the community is finally shifting towards the study of such networks. 

The discussion above raises the natural question whether there exists an optimal metric, at least for understanding of the problems studied in the present paper. The somewhat disappointing (but expectable) answer is that there is no ``best'' metric fitted for the study of all networks, not even for the study of a certain type of discrete Ricci curvature and its flow. 
From the empirical networks viewpoint this fact is easily understandable: Each community has developed its specific metric or set of metrics which best serve its needs, in the sense that they closely model the type of network that represents the object of interest for that group. (In certain settings, a natural metric, up to minor variations, imposes itself, such in the case of Imaging -- see the relevant examples in the sequel. The same holds true for electric networks, where the resistance metric is the default one.) These ad hoc metrics might differ widely from the standard, mathematically motivated ones mentioned above and, as such, they do not necessarily behave as well in concordance with the various type of curvatures. 
From a mathematical viewpoint, it is important to remember that the discrete Ricci curvatures captures different aspects of the classical notion (the Forman-Ricci curvatures vs. the Haantjes curvature), or include higher dimensional aspects of the networks (full Forman and Haantjes, vs. graph Forman). Therefore, their behaviors in conjunction with the various metrics also diverge widely. It is, therefore, advisable to understand their comparative behavior using the same background metric. This is the approach adopted herein, where the standard path metric was employed.  Note that using that, when using a common metric, various discretization capture the same essential behavior of evolving networks -- see \cite{SPRSJC}. 
Moreover while from a mathematician's viewpoint, the degree path metric is apparently ideal (as the number of studies employing it seems to suggest), in practice its connection with a specific type of discrete curvature, e.g. Olliver-Ricci curvature, is far less strong than expected. In fact, for certain types of Semantic Networks, the probabilistic version of the resistance metric (see Appendix 3) is by far the best suited \cite{CMS21}.  
Thus the problem of determining the best suited metric to be used with a specific discrete Ricci curvature, for each type of network, deserves an extensive experimental study, whose scope and breadth are far beyond those of the present paper. 
\end{rem}



Yet another -- and simpler -- generalization of Ricci curvature to metric measure spaces has been devised, based on the works of Bakry, Emery and Ledoux \cite{BE, BL}. 
More specifically, they consider {\it manifolds with density}, i.e. Riemannian manifolds $M^n$, additionally endowed, with a smooth, positive density function $\Psi = \Psi(x)$, that induces weighted $n$- and $(n-1)$-volumes, e.g. in the classical cases $n = 2$ and $n = 3$, volume, area and length. More precisely, the  volume, area and length elements $dV, dA, ds$ of the weighted manifold $(M^n,\Psi)$ are given by:
\[dV = \Psi dV_0, dA = \Psi dA_0, ds = \Psi ds_0\,,\]
where $dV_0$ represents the natural (Riemannian) volume element of $M^n$, etc. Usually 
density functions of the type $\Psi(x) = e^{-\varphi(x)}$ are considered.
(However more general density functions have also been studied -- see \cite{Mo}.) 


The  Bakry, Emery and Ledoux generalization to manifolds with density of Ricci curvature is defined as:
\begin{equation} \label{eq:K-BEL}
{\rm Ric_\varphi} = {\rm Ric} + {\rm Hess}\varphi\,,
\end{equation}
(where {\rm Hess} denotes the Hessian matrix).
It is important for us to note that, for surfaces,
Ricci curvature reduces, essentially, to Gaussian curvature $K$, more precisely $K = \frac{1}{2}{\rm Ric}$. 
Another, closely related, but perhaps more intuitive, generalization of Gaussian curvature for weighted surfaces is due to Corwin et al. \cite{Co+}, namely:
\begin{equation} \label{eq:K-Corwin}
K_\varphi = K + \Delta\varphi\,,
\end{equation}
where $\Delta\varphi$ denotes the Laplacian of $\varphi$. It should be stressed that this represents a natural generalization: It reduces to the usual Gaussian curvature (up to a multiplicative constant) for $\varphi \equiv {\rm const.}$ and, moreover, it also satisfies a generalized Gauss-Bonnet Theorem.
%
The reader should note that, unlike Morgan \cite{Mo}, but following other authors, and, moreover, in concordance with Forman's work, we adopt  here the ``+'' convention for the sign of the Hessian and Laplacian commonly, since this is more intuitive, at least in the context of Imaging where weights, that is grayscale values are always positive. 
This simpler approach can be indeed applied to the sampling of images, where grayscale value can be interpreted as a measure (distribution) over the pixels' grid \cite{LLZS}. This fact further encourages us to extend the Ricci curvature-based sampling to Complex Networks.

\section{Discrete Ricci Curvature}

There are several ways of incorporating both node and edge weights into a comprehensive Ricci curvature for networks. 
Since a full explanation of the ideas and techniques employed in devising the notions below would take us too far afield, and, furthermore, we have detailed this in previous works and, moreover, in the present paper we employ these discrete notions of Ricci curvature only in the practical, network context, we do not develop these definitions here, but rather restrict ourselves essentially to bringing the relevant formulas.

The simplest (from a computational viewpoint) type of network Ricci curvature is the 1-dimensional (graph) version introduced in \cite{SMJSS}  of Forman's Ricci curvature \cite{Fo}, originally devised for weighted $CW$ complexes of dimension  $\geq 2$. It is derived from the classical  {\it Bochner-Weitzenb\"{o}ck formula} of Riemannian Geometry \cite{Jo},  that connects the {\it Bochner} (or {\it Hodge}) {\it Laplacian} on a manifold and its various curvatures, namely 
\begin{equation} \label{saucan-eqn:1}
\Box_p = dd^* + d^*d = \nabla_p^*\nabla_p  + {\rm Curv}(R)\,,
\end{equation}
where $\Box_p$ denotes the  {\it Riemann-Laplace operator} $\Box_p$ on $p$-forms, $\nabla_p^*\nabla_p$ is the {\it Bochner} (or {\it rough}) {\it Laplacian} and ${\rm Curv}(R)$ is an 
expression of the {\it curvature tensor} with linear coefficients  
where $\nabla_p$ denotes the {\it covariant derivative} operator.

Forman \cite{Fo} demonstrated that an analogue of the Bochner-Weitzenb\"{o}ck formula holds in the general setting of CW complexes, of which graphs and polyhedra are particular cases. More precisely, he showed that there exists a canonical decomposition of the form:
\begin{equation} \label{saucan-eqn:3}
\Box_p = B_p + F_p\,,
\end{equation}
where $B_p$ is a {\it non-negative operator} and $F_p$ is a diagonal matrix. $B_p$ and $F_p$ are called,
in analogy with the classical Bochner-Weitzenb\"{o}ck formula, the {\it combinatorial Laplacian} and the {\it
combinatorial curvature function}, respectively. 

%
 In particular, for $p = 1$, we obtain
\begin{align}\label{eq:bochner1}
\Box_1p(e) = \Box_1(e,e) = \frac{w(v_1)}{w(e)} + \frac{w(v_2)}{w(e)}\,;
\end{align}
where $v_1,v_2$ are the end nodes of the edge $e$, and $w(v_1),w(v_2),w(e)$ represent their respective weights.  (For the definition in the general case, see Appendix 2.)

Then, given a $p$-dimensional cell $\alpha = \alpha^p$,  one defines the {\it
curvature function} $F_p: C_p \rightarrow C_p$
\begin{eqnarray}
\mathcal{F}_p(\alpha) = \langle F_p(\alpha),\alpha \rangle. 
\end{eqnarray}
In the special case of dimension $p=1$ one defines, by analogy with the classical case, the
 discrete (weighted) {\it Forman-Ricci curvature} on $\alpha = \alpha^1$, i.e. 1-cells (edges), namely 
\begin{eqnarray}
\label{saucan-def:Ricci}
{\rm Ric_F}(\alpha)= \mathcal{F}_1(\alpha).
\end{eqnarray}

Before proceeding further, let us note the similarity between Formula (\ref{eq:K-Corwin}) and Forman's Formula  (\ref{saucan-eqn:3}), and the (essential) identification of Gaussian and Ricci curvature in the surface case, further suggest Forman's Ricci curvature as a possible sampling tool. 
Furthermore,  this similarity suggests that theleft-hand term $\Box_1$ can be interpreted not just like a ``corrected'' Laplacian, but also as a ``adjusted'' curvature. Thus, one is conducted towards a possible sampling of networks based on Forman-Bochner Laplacian, instead of its Ricci curvature counterpart.
This coupling with a Laplacian, that opens further directions for possible applications, represents yet another advantage, besides its computational simplicity, of Forman's curvature over Ollivier's one.

While the formula of ${\rm Ric_F}$ in the case of generic $n$-dimensional $CW$ complexes is quite complicated (see Appendix 2), in the 1-dimensional case, that is of graphs/networks, it reduces to the following elementary formula:
\begin{equation}	\label{FormanRicciEdge}
\mathbf{F}(e) = w(e) \left( \frac{w(v_1)}{w(e)} +  \frac{w(v_2)}{w(e)}  - \sum_{e(v_1)\ \sim\ e,\ e(v_2)\ \sim\ e} \left[\frac{w(v_1)}{\sqrt{w(e) w(e(v_1))}} + \frac{w(v_2)}{\sqrt{w(e) w(e(v_2))}} \right] \right)\,;
\end{equation}
%
where $\mathbf{F}$ denotes the Forman-Ricci curvature for networks\footnote{notice the change of notation for clarity reasons}, and where 
\begin{itemize}
	\item $e$ denotes the edge under consideration between two nodes  $v_1$ and  $v_2$;
	\item  $w(e)$ denotes the weight of the edge  $e$ under consideration;
	\item  $w(v_1), w(v_2)$ denote the weights associated with the nodes  $v_1$ and  $v_2$, respectively;
	\item  $e(v_1) \sim e$ and  $e(v_2) \sim e$ denote the set of edges incident on nodes  $v_1$ and  $v_2$, respectively, after \textit{excluding} the edge $e$ under consideration which connects the two nodes   $v_1$ and  $v_2$, i.e. $e(v_1),e(v_2) \neq e$. 
\end{itemize}

While, as discussed above, we are mainly interested in edge-centric measures, and more specifically in Ricci curvature, one can also define the Forman-scalar curvature, in manner similar to the $PL$ manifolds case pioneered by Stone \cite{St} (see also \cite{GS}), to be
\begin{equation} 
\kappa_{{\bf F}(v)} = {\bf F}(v) 
= \sum_{e_k \sim v} {\rm Ric}_F(e_k)\,;
\end{equation}
where $e_k \sim v$ denotes  the edges $e_k$ adjacent to the vertex $v$. 

For 2-dimensional complexes the formula of the Forman-Ricci curvature is only slightly more complicated:
\begin{equation} \label{eq:Forman-2d}
\hspace*{-2.65cm}{\rm Ric}_{\rm F} (e) = w(e) \left[ \left( \sum_{e \sim f} \frac{w(e)}{w(f)}+\sum_{v \sim e} \frac{w (v)}{w(e)}	\right) \right. 
\end{equation}
\[\hspace*{2.65cm}
- \left. \sum_{\hat{e} \parallel e} \left| \sum_{\hat{e},e \sim f} \frac{\sqrt{w(e) \cdot w(\hat{e})}}{w(f)} - \sum_{v 	\sim e, v \sim \hat{e}} \frac{w(v)}{\sqrt{w(e) \cdot w(\hat{e})}} \right| \right] \; ;
\]
where ``$\hat{e} \parallel e$" denotes that the edges $\hat{e}$ and $e$ are parallel, that is they can both belong to a common 2-dimensional face, or have a common vertex, but not both, simultaneously; and where ``$\sim$" has the same significance as in the previous formula (i.e. incidence).

There are two special cases extremely important in applications. The first case is that of square grids, as they arise most naturally in Imaging. Due to the specific (and evident) parallelism relationship, the resulting formula attains the following simple form:
\begin{equation} \label{eq:Ricci-Forman2D}
\hspace*{-0.5cm}{\rm Ric}_F(e_0) =
w(e_0)\left[\left(\frac{w(e_0)}{w(c_1)} +
\frac{w(e_0)}{w(c_2)}\right) -
\left(\frac{\sqrt{w(e_0)w(e_1)}}{w(c_1)} +
\frac{\sqrt{w(e_0)w(e_2)}}{w(c_2)}\right)\right]\,,
\end{equation}

The second special case is that of $PL$ manifolds and, more generally, of graphs and networks where the only 2-cycles are triangles. This type of structure is, again, relevant in Graphics and related fields. Moreover, it is precisely the type of graph (network) that we developed and discussed above. 
The fitting formula is
\begin{equation} \label{eq:Forman-2d-tr}
{\rm Ric}_F(e) = w(e)\left[\sum_{t > e}\frac{w(e)}{\omega(t)} + \left(\frac{w(v_1)}{w(e)} + \frac{w(v_2)}{w(e)}\right) \right.
\end{equation}
\[
\hspace*{4cm}  \left. -  \sum_{\tilde{e} \sim e}\left|\sum_{t > e,\, t > \tilde{e}}\frac{\sqrt{w(e)w(\tilde{e})}}{w(t)} - \sum_{v < e,\, v < \tilde{e}}\frac{w(v)}{\sqrt{w(e)w(\tilde{e})}} \right| \right] \, .
\]
where ``$t > e$" denotes that the edge $e$ is a face of the triangle (2-cycle) $t$.

To differentiate between the 1- and 2-dimensional versions of Forman-Ricci curvature, we shall refer to the former as the 
{\it graph} or {\it reduced Forman-Ricci curvature}, and to the later as the {\it full Forman-Ricci curvature} (or simply the 
{\it  Forman-Ricci curvature}). 

Again, we can define the fitting scalar curvature, in a similar manner:
\begin{equation} 
\kappa_{{\rm Ric}_F}(v) = {\rm Ric}_F(v) = \sum_{e_k \sim v} {\rm Ric}_F(e_k)\,;
\end{equation}
where, again, $e_k \sim v$ stands for all the edges $e_k$ adjacent to the vertex $v$.

\begin{rem}	
Note that, while we concentrated above on Forman's Ricci curvature, one can substitute instead Ollivier's discretization \cite{Ol1}, \cite{Ol2}, as it is widely employed in Network Theory (see, e.g. \cite{WJB}, \cite{Ni2015}, \cite{Ni2019}, \cite{Allen1}, \cite{Allen2}). We preferred Forman's curvature for a number of reasons, perhaps not the least of them being its clear computational advantages, as well as for the easy manner in which it extends to higher dimension. An additional advantage we already mentioned above, namely that by its defining formula  it comes coupled with a Laplacian, also allows us to explore new directions of study -- See Section \ref{sec:Kernels} below.
\end{rem}

The third type of discrete Ricci curvature that we have considered in our experiments is one that we recently introduced \cite{SSJ}, \cite{SSJ1}, namely the Haantjes-Ricci curvature \cite{Ha}. In contrast with Forman-Ricci curvature, whose definition is based on the discretization of and advanced techniques in Riemannian Geometry, namely the {\it Bochner-Weitzenb\"{o}ck formula} (see for instance \cite{Jo}), Haantjes curvature is derived from a purely metric notion \cite{Ha} devised originaly to study curves in metric spaces. Before we can define this new type of Ricci curvature, we first have to introduce the  curvature of paths and cycles:

Given a curve $c$ in a metric space $(X,d)$, and $p, q, r$ 
points on $c$, $p$ between $q$ and $r$, the Haantjes curvature of the curve $c$, at the point $p$ is defined as
\begin{equation}                         
\label{eq:haantjes-1}
\kappa_{H}^2(p) = 24\lim_{q,r \rightarrow p}\frac{l(\widehat{qr})-d(q,r)}
{\big(d(q,r)\big)^3}\,;
\end{equation}
where $l(\widehat{qr})$ denotes the length, in the intrinsic metric induced by 
$d$, of the arc $\widehat{qr}$. In the network case, 
$\widehat{qr}$ is replaced by a path $\pi = v_0,v_1,\ldots,v_n$, and the 
subtending chord by edge $e = (v_0,v_n)$. When discretizing this notion we should start from the following two observations: (a) The limiting process has no 
meaning in this discrete case; (b) The normalizing constant ``24'' which 
ensures that, the limit in the case of smooth planar curves will coincide with 
the classical notion, is redundant in the discrete setting setting. We are thus led to the following 
definition of the \textit{Haantjes curvature of a simple path} $\pi$:
\begin{equation}                         
\label{eq:haantjes-path}
\kappa_{H}^2(\pi) = \frac{l(\pi)-l(v_0,v_n)}{l(v_0,v_n)^3}\,;
\end{equation}
where $l(v_0,v_n) = d(v_0,v_n)$, where, as above, $e = (v_0,v_n)$ represents an edge and where $l$ denotes the length, in the considered metric $d$. (If no specific metric is given, one can always metricize an edge weighted graph using the standard {\it path metric} -- see, e.g. \cite{DD}.) In particular, in the case of the combinatorial metric, we obtain that, for path $\pi = v_0,v_1,
\ldots,v_n$ as above, $\kappa_H(\pi) = \sqrt{n-1}$. It is important to note that considering simple 
paths does not represent any restriction, since a metric arcs are, by definition, simple 
curves. However, to capture in the discrete context the local nature of the Ricci 
(and scalar) curvature, as well as to make computations feasible, we shall restrict to paths $\pi$ such that $\pi^{*} = v_0,
v_1,\ldots,v_n,v_0$ is an {\it elementary cycle}.

In the case of general edge weights, one can not apply directly the Haantjes curvature, since in this case the network is not necessarily a metric space, given that the total weight $w(\pi)$ of a path $\pi = v_0,v_1,\ldots,v_n$ is not 
necessarily smaller than the weight of its subtending chord $e=(v_0,v_n)$. In fact, a graph endowed with general weighted graphs will not constitute a metric space, since such weights can (and usually will) fail the triangle inequality. 
However, it is still possible to use the Haantjes curvature, and we even can turn this to our own 
advantage,  by reversing the roles of $w(\pi)$ and $w(v_0,v_n)$ in the definition of
the Haantjes curvature and assigning a minus sign to the curvature of cycles for 
which this occurs. In consequence, this approach allows us to define a variable sign 
Haantjes curvature of cycles (hence, as it will become clear below, a Ricci curvature as well), even if the given 
network is not a directed one.

Before proceeding further, it should be noted that the case when $w(v_0,v_1,\ldots,v_n) = w(v_0,v_n)$, i.e., that of zero 
curvature of the 2-cell $\mathfrak{c}$ with $\partial \mathfrak{c} = v_0,v_1,\ldots,v_n,v_0$,  
straightforwardly corresponds to the splitting case for the path metric induced by 
the weights $w(v_i,v_{i+1})$.

We can now define Haantjes-Ricci curvature 
in a straightforward manner,  as
\begin{equation}
\label{eq:haantjes-compute}
	\kappa_{H}(e) = {\rm Ric}_{H}(e) = \sum_{\pi \sim e}\kappa_{H}(\pi)\,;
\end{equation}
where $\pi \sim e$ denote the paths that connect the vertices adjacent to the edge $e$.

Moreover, as for the Forman-Ricci curvatures, one can also define the Haantjes-scalar curvature:
\begin{equation} 
	\kappa_{H}(v) = {\rm scal}_{H}(v) = \sum_{e_k \sim v} {\rm Ric}_{H}(e_k)\,;
\end{equation}
where $e_k \sim v$ stands for all the edges $e_k$ adjacent to the vertex $v$. 

\begin{figure}[htb]
	\begin{center}
		\includegraphics[scale=2]{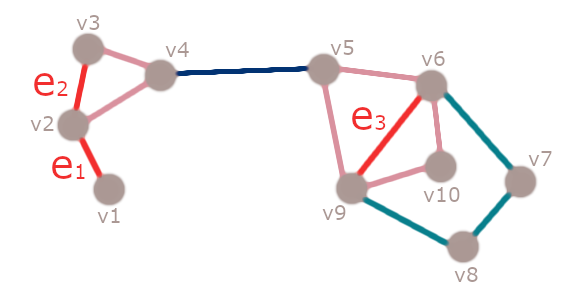}
		%
		%
	\end{center}
	\caption{A small combinatorial network (i.e. with all vertex and edge weights equal to 1) based on the well known ``Dolphins'' social network \cite{Kun}, emphasizing the differences between the three types of Ricci curvature considered. ${\rm Ric}_H(e_1)$ is zero, since there are no cycles (faces) adjacent to $e_1$ and  ${\bf F}(e_1)$ also equals $0$ (as it is really seen from the degrees of its end vertices) and, moreover, so does ${\rm Ric}_F(e_1)$, since in this case it coincides, due to the absence of adjacent faces, with ${\bf F}(e_1)$. 
	 ${\rm Ric}_H(e_2) = \sqrt{2}$, given that there is only one face -- a triangle -- adjacent to $e_2$. Degree counting easily renders ${\bf F}(e_2) = -1$, and ${\rm Ric}_F(e_2) = 2$. In the case of $e_3$, degree counting renders ${\bf F}(e_3) = -4$. Since the edge $e_3$ is adjacent to 2 triangles and a quadrangle, ${\rm Ric}_H(e_3) = 2\sqrt{2} + \sqrt{3}$, while given the fact that there is only one edge parallel to $e_3$, namely $(v_7v_8)$, we obtain that ${\bf F}(e_2) = -4$. 
    }
\end{figure}

\subsection{Curvature Based Sampling}

As we have already emphasized in the introduction, Ricci curvature allows us not only to focus on the edge-centric study of networks, but also to analyze the distribution and role of higher order correlation in networks. It follows, therefore, that an edge-based sampling is the natural one, if one wishes to explore the edge and higher order correlations structure. Nevertheless, the classical graph-theoretical approach to networks is still relevant, thus we also explore the vertex-based networks' sampling and compare it to the edge based one.

We first exemplify the network sampling using all the types of curvatures introduced above on two examples of real-life weighted networks:\footnote{We restrict only to these so not to overextend the experiments section.} The classical ``Kangaroo'' one, which due to its small size allows one to better relate the computations to the network structure, and the increasingly popular ``Les Mis\'{e}rables'' one, describing the relationships between the characters in the classical Hugo novel. The data for all the networks, except that for the ``C. elegans'', was downloaded from the KONECT website \cite{Kun}; in the later case the source is \cite{Cho++}. For each network and each type of curvature, both edge (Ricci curvature) based and vertex (scalar curvature) are included. In the original network, Ricci curvature is depicted according to the standard method, i.e. edge thickness being proportional to the absolute value of the curvature, whereas sign is indicated using the (standard) color code: Red color for negative curvature and blue color for positive one. All the considered networks are weighted ones -- for the nature of the specific weights the reader is referred to \cite{Kun}, \cite{Cho++} and the references therein. 
The sampling method adopted here is the simplest, yet widely adopted one of high-pass filtering the curvature with high absolute curvature. More precisely, for a prescribed percentage of absolute curvature, the vertices with curvature below the chosen threshold are removed, together with the edges adjacent to them, in case that vertex-based sampling is adopted. I the case when edge-based sampling is chosen, the procedure is similar, with the evident difference that, in this case, edges are removed. Note that in this case the ends of the deleted edges are not removed, except in the case of so called leaves (i.e. edges with at least one vertex of degree 1). The percentages differ from a network to another and from a type of curvature to another and they were chosen such that approximatively the same number of vertices would be retained. 
%
%
\begin{figure}[htb]
	\begin{center}
		\hspace*{-1.cm}\includegraphics[scale=0.23]{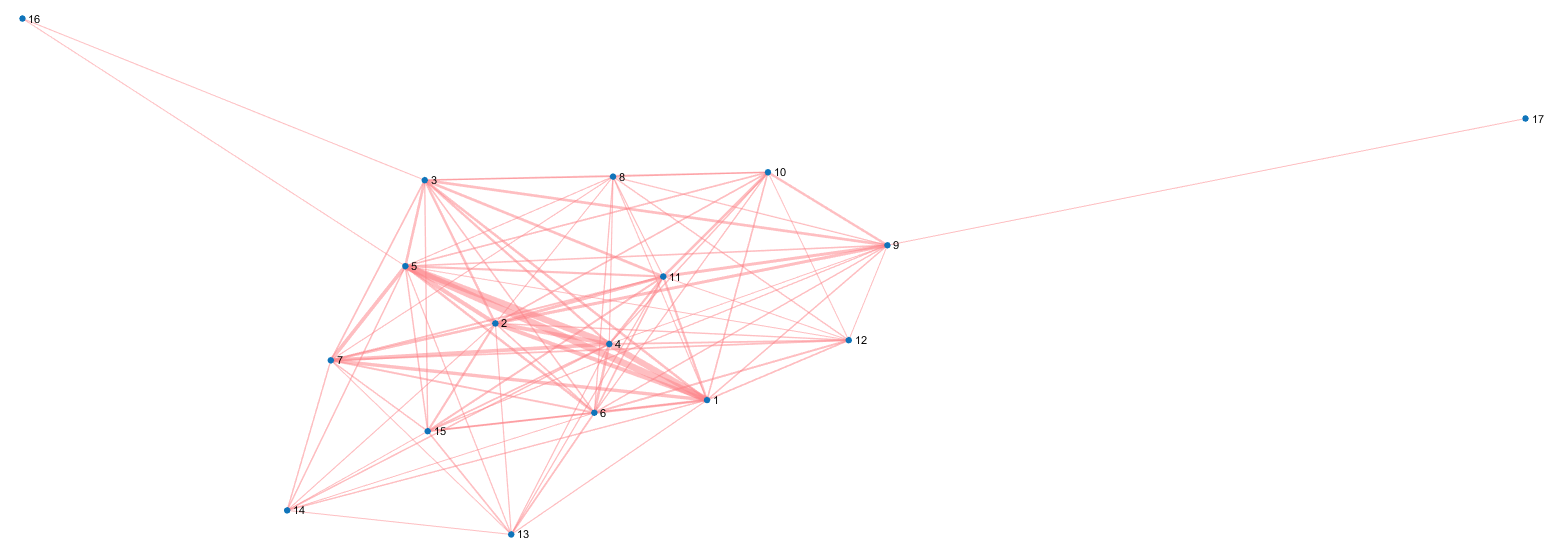} 
		\includegraphics[scale=0.375]{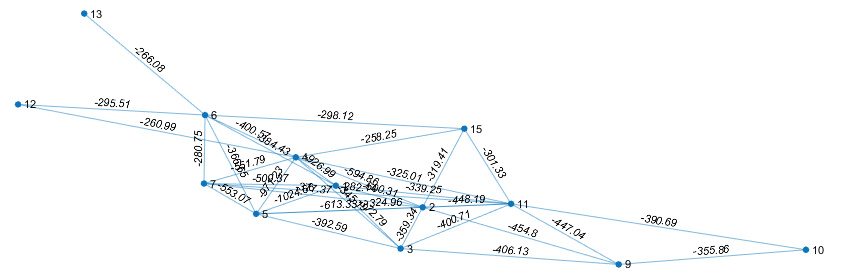}		\hspace*{-2.25cm}\includegraphics[scale=0.25]{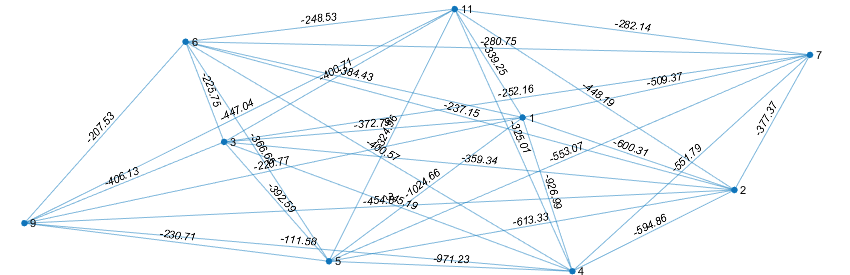}
		%
		%
	\end{center}
	\caption{The graph Forman-Ricci curvature based sampling by vertices (middle) and by edges (below) of a Kangaroo social network (above).  Here 40$\%$ of the edges were retained, yet the main features of the image is still clearly visible. }
\end{figure}


\begin{figure}[htb]
	\begin{center}
		\hspace*{-2.2cm}\includegraphics[scale=0.23]{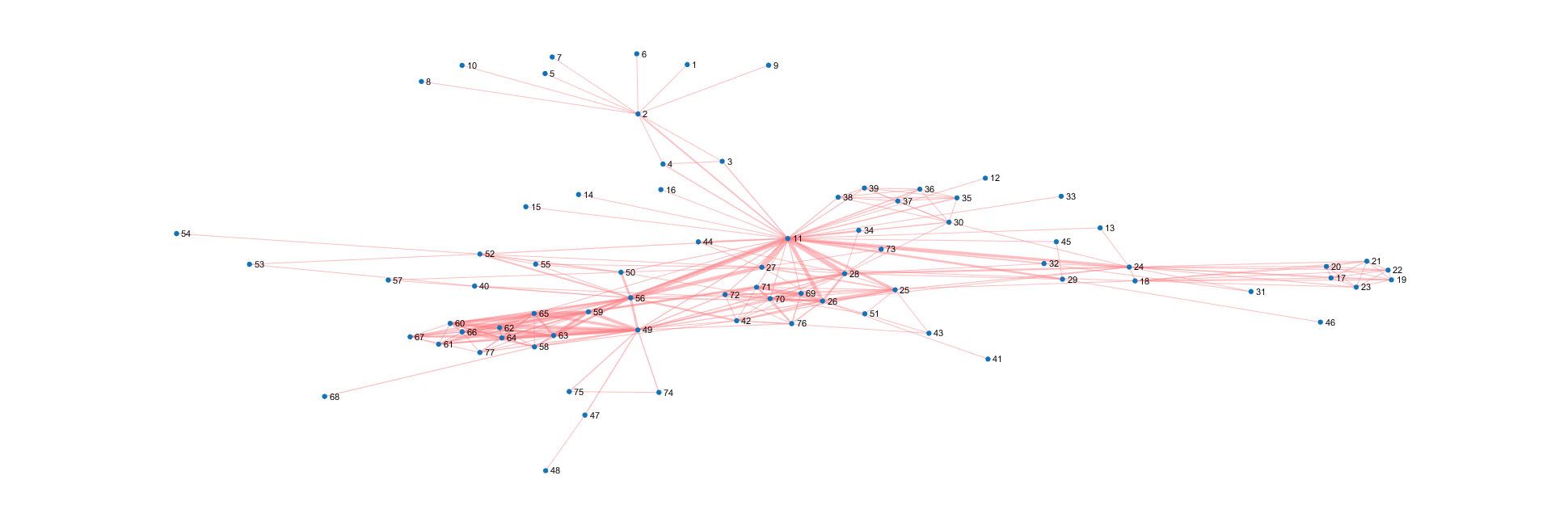} 
		\includegraphics[scale=0.375]{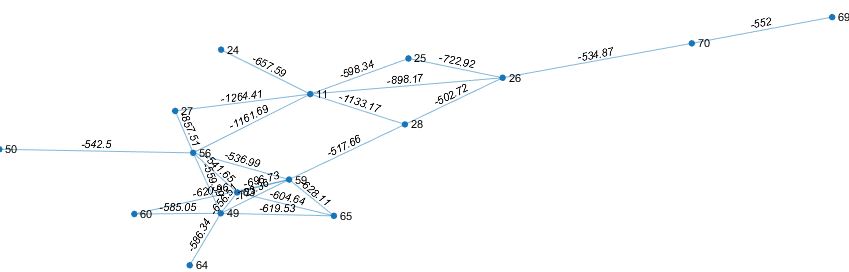}
		\includegraphics[scale=0.375]{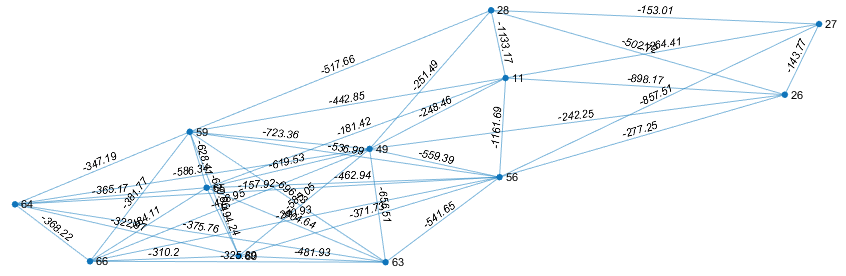}
		%
		%
	\end{center}
	\caption{The Graph Forman-Ricci curvature edges-based sampling (middle) and by vertices-based (below) 
		of the {\em Les Mis\'{e}rables}  social network (above).
		}
\end{figure}

\begin{figure}[htb]
	\begin{center}
		\includegraphics[scale=0.65]{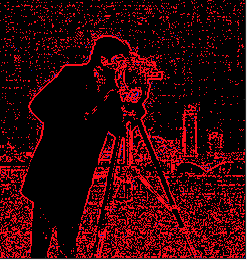}
		\includegraphics[scale=0.66]{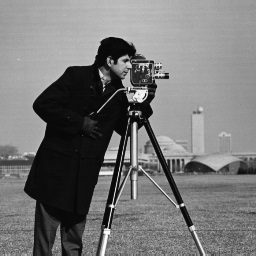}
		%
		%
	\end{center}
	\caption{The Graph Forman-Ricci curvature based sampling (left) of a classical test image (right).  Here 20$\%$ of the edges were retained, yet the main features of the image is still clearly visible. 
		Note the mostly red coloring of the resulting curvature images, showing that, except at a sparse set of pixels, the graph Forman-Ricci curvature is negative.}
\end{figure}


Ideally one should include in the computation of the full Forman-Ricci curvature all the cycles. However, determining them is a computationally costly operation, especially so since higher order cycles (faces) are quite common in real-life networks \cite{WSJ2}. Nevertheless, in the study of real-life networks, one can discard, at least in first approximation, the higher order faces. Indeed, both in social and biological networks, higher order cycles represent weaker connections, that should be taken, at best, with a lower weight signifying their reduced contribution (or probability of existence). In fact, and due partially to the powerful theoretical tools available in this case, the study of higher order correlations is commonly restricted only to order 3 ones (i.e. triangles). Given these considerations, and the fact that we are here only proposing a new paradigm for networks sampling, rather than concentrating for the extensive study of a specific (type of) network, we restrict our computations to 3- and 4-cycles, that is to faces that are triangles or quadrangles. 

\begin{figure}[htb]
	\begin{center}
		\hspace*{-2.25cm}\includegraphics[scale=0.23]{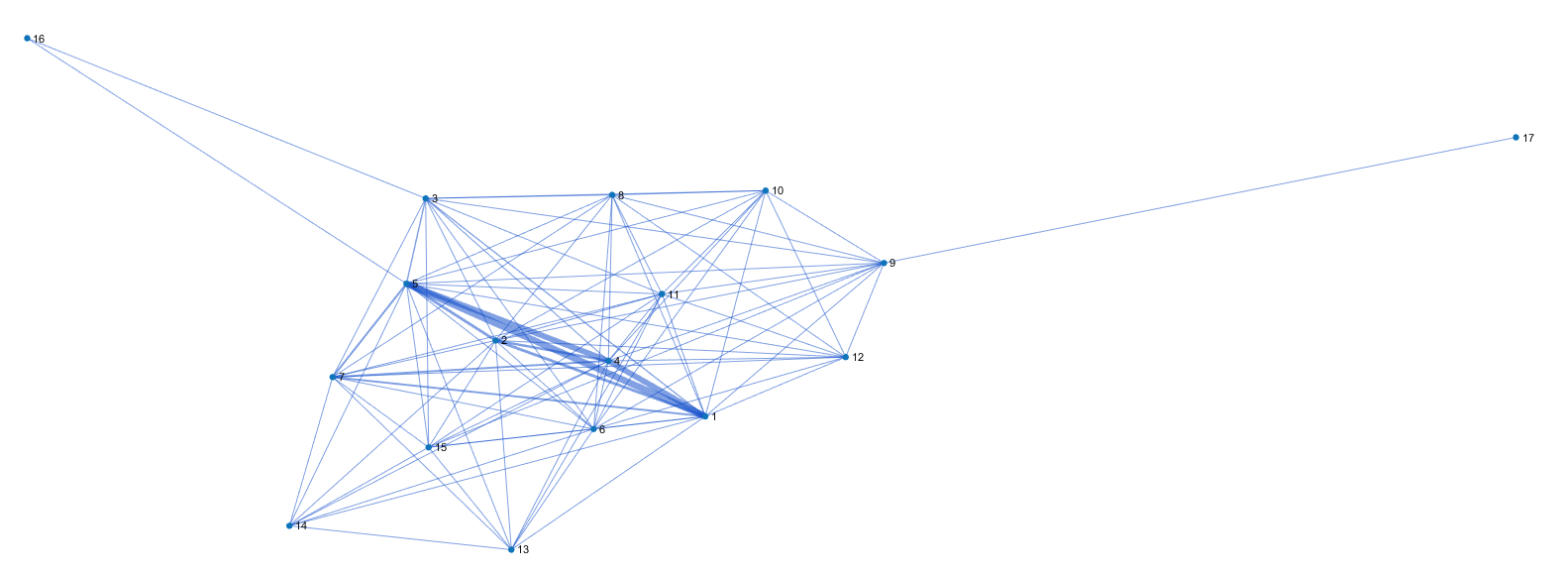}
		\includegraphics[scale=0.375]{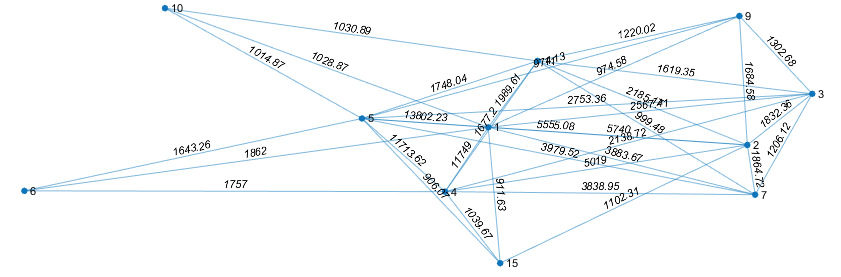}		\includegraphics[scale=0.375]{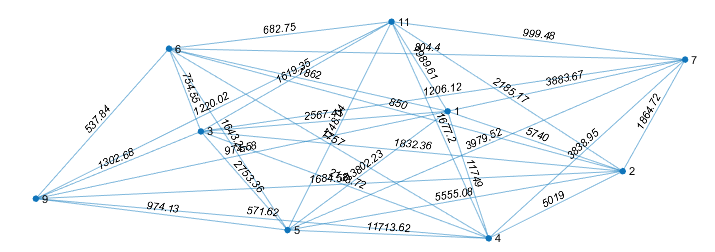}
		%
		%
	\end{center}
	\caption{The full Forman-Ricci curvature-based sampling by vertices (middle) and by edges (below) of a Kangaroo social network (above).  Here 40$\%$ of the edges were retained.}
\end{figure}

\begin{figure}[htb]
	\begin{center}
		\hspace*{-2.25cm}\includegraphics[scale=0.23]{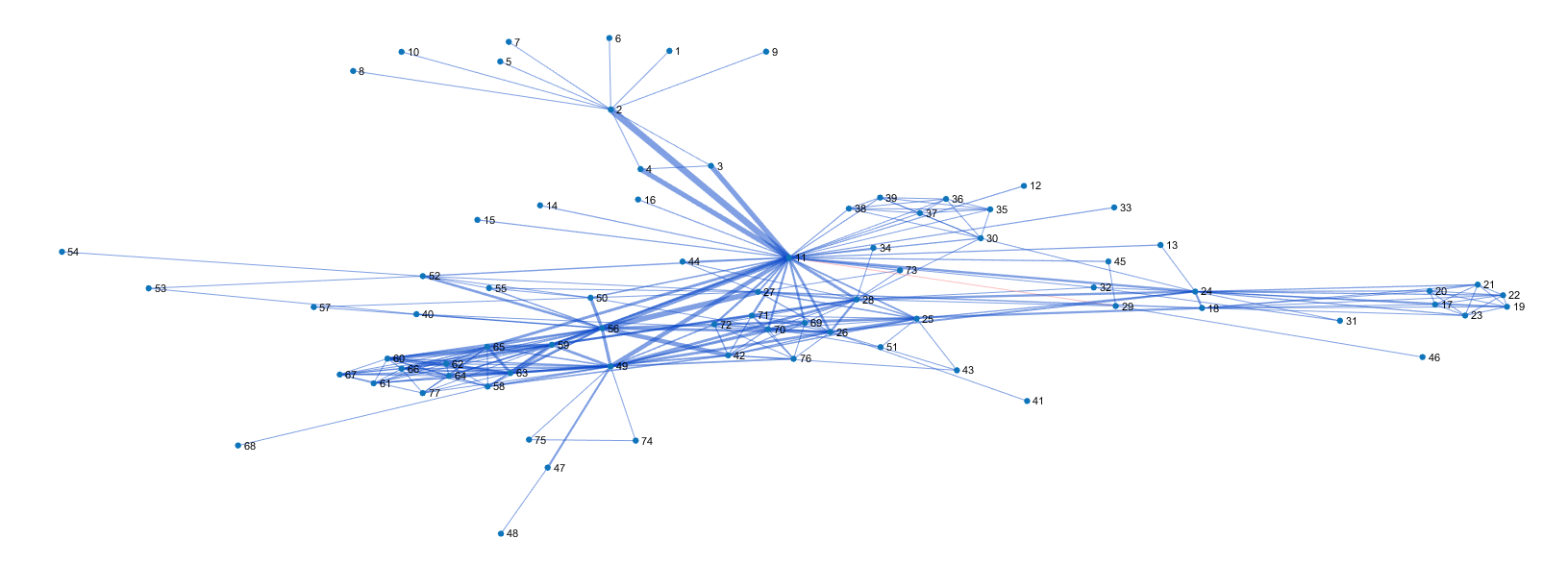}
		\includegraphics[scale=0.375]{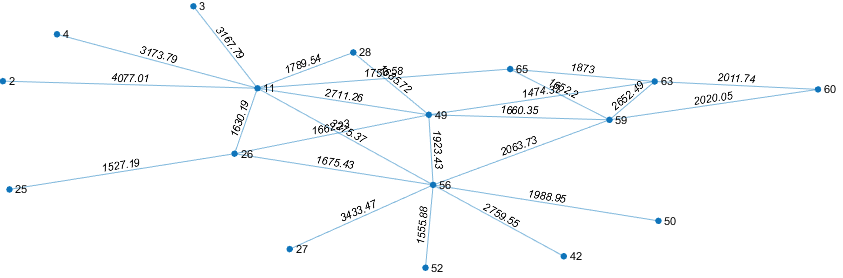}		\includegraphics[scale=0.375]{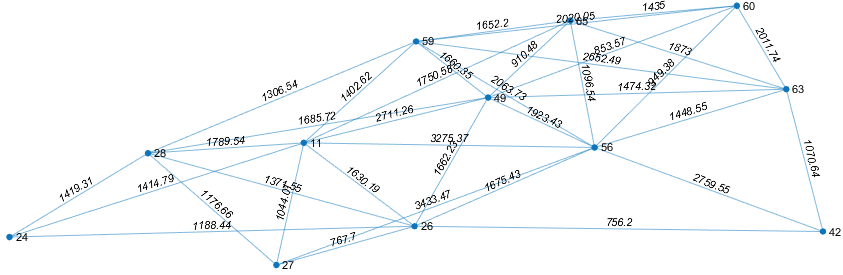}
		%
		%
	\end{center}
	\caption{The full Forman-Ricci curvature-based sampling by vertices (middle) and by edges (below) of the Les Mis\'{e}rables social network (above).  Here 10$\%$ of the edges were retained.} 
\end{figure}

The results of the vertex- and edge-based sampling, using the various considered Ricci curvatures, of a couple of well-known social networks are depicted  in Figures 2-4 and 5-7. 
Let us note that at least for medium sized and large scale networks, the ``common core'' of obtained by the intersection of the sampling results using the different curvatures can be quite small, illustrating the fact that each of the curvatures captures quite different different properties of the network, as it is the case with the various discretizations of any classical curvature. Thus the core represents the subgraph that comprises the essential geometric properties of a network. 
Note also that for smaller networks and/or larger percentage of remaining nodes, the cores are both larger and closer to each of the sampling results. 

\begin{figure}[htb] \label{fig:LesMis-Comon}
	\begin{center}
		\includegraphics[scale=0.23]{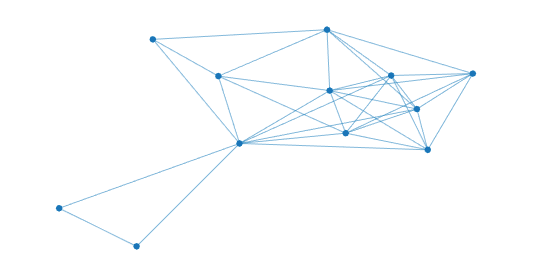}\includegraphics[scale=0.23]{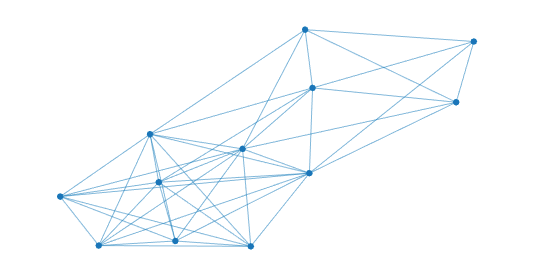}\includegraphics[scale=0.23]{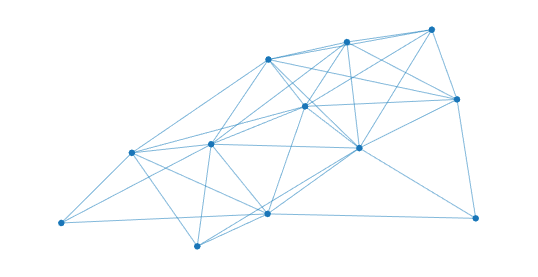}
		
		\includegraphics[scale=0.23]{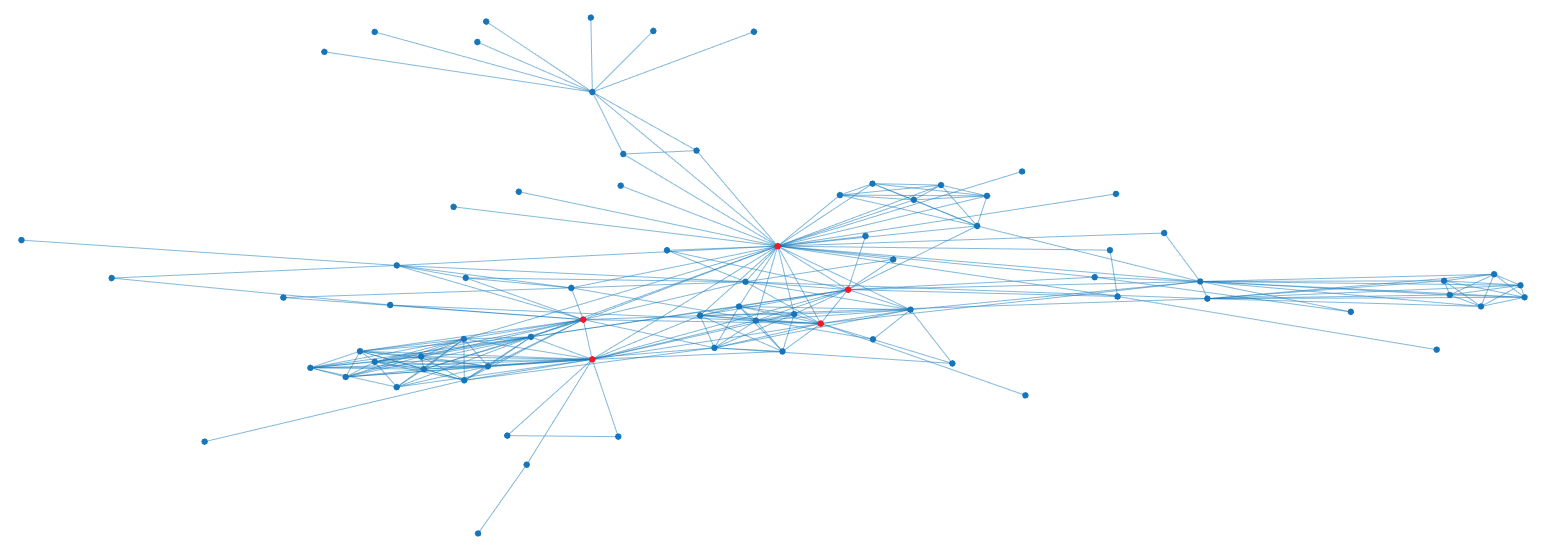}
		%
		%
	\end{center}
	\caption{The common ``core'' (below), i.e. set of vertices of the ``Les Mis\'{e}rables'' network after sampling by all of the curvatures considered (above, from left to right: Haantjes, graph Forman, full Forman). In each of the cases $15\%$ of the original vertices were retained. Note that the resulting, sampled networks are widely different. Moreover, note that the common nodes (in red, below) is quite restricted. This is a direct consequence of the fact that different discretizations of Ricci curvature capture diverse properties of the classical notion, thus rendering a small set of nodes and edges that capture the most essential geometric properties of the network.}
\end{figure}

In addition to these real life networks we are also including the example of a regular square grid as it arises naturally in Imaging -- See Figures 4, 8 and 9. More precisely, nodes denote centers of pixels and edges connect adjacent pixels, edge weights (lengths) are given by the differences in height, that is gray scale intensity of the image, and areas of the resulting 2-cells (quadrangles) are obtained, via Heron's formula from the edge lengths. (For details and illustrations see \cite{BCLS}.) The color scheme adopted for sign depiction is the same as before. 
Given the construction of the network and its weights, only vertex sampling sampling is considered. 
As one can see from the results depicted in Figure 4, when the sampling was applied to a classical test image, even when retaining only 25$\%$ of the original nodes (pixels) the original image is still quite easily recognizable and, moreover, curvature does indeed function as an edge detector, as expected. (See also \cite{Sa11-1} and the references therein.)

\begin{figure}[htb]
	\begin{center}
		\includegraphics[scale=0.65]{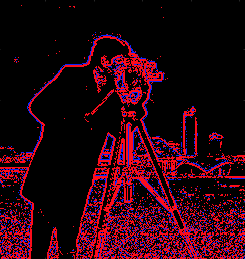}
		\includegraphics[scale=0.66]{7-0.jpg}
		%
		%
	\end{center}
	\caption{The full Forman-Ricci curvature based sampling of the ``Cameraman'' (left) compared to the graph Forman-Ricci curvature based one (right).  Here, again, 20$\%$ of the edges were retained. 
		Note that both blue and red edges are visible, showing that full Forman-Ricci curvature takes both positive and negative values.}
\end{figure}

Besides the two versions of Forman-Ricci curvature we experimented with the Haantjes-Ricci curvature. 
Again, as in case of the full Forman-Ricci curvature (and for the same reasons), we restrict ourselves to 3- and 4-cycles. 
In addition to its applicability to networks, it is a natural curvature measure for images \cite{Sa17}. However, in this case the obtained results are not as promising as those obtained using the full Forman-Ricci curvature, given than for images, that is for surfaces embedded in $\mathbb{R}^3$, it takes only positive values -- see Figure 9.

\begin{figure}[htb]
	\begin{center}
		\includegraphics[scale=0.425]{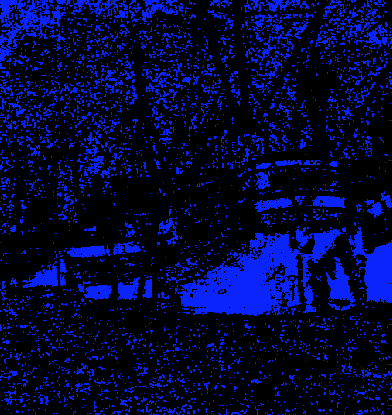}
		\includegraphics[scale=0.575]{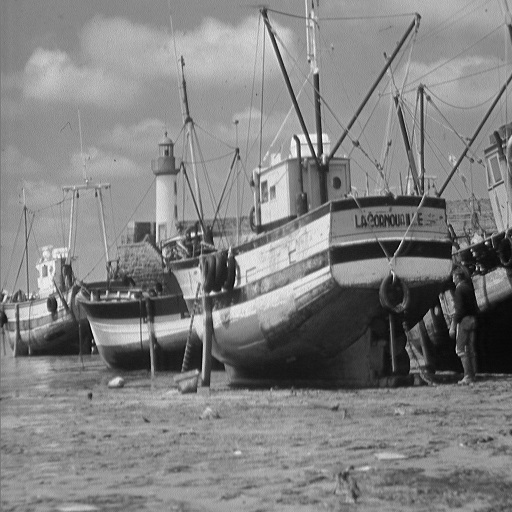}
		%
		%
	\end{center}
	\caption{The Haantjes-Ricci curvature based sampling (left) of a classical test image (right).  Here again 20$\%$ of the edges were retained and the main features of the image are still clearly visible. However, it is evident that Haantjes curvature is outperformed, in the case of images, by both Forman curvatures.
		Note the blue coloring of the resulting curvature image, showing that  the Haantjes-Ricci curvature is positive, a fact that follows from it being computed using the usual distances in Euclidean space.}
\end{figure}


\subsubsection{Ricci flow}

Besides the simple, direct approach to sampling that filters out a certain, prescribed percentage of edges, a more automatic -- and with far deeper theoretical motivation -- exists, namely the one based on the {\it Ricci flow} \cite{WSJ1}, \cite{WSJ2}, \cite{SSAZ}. This approach not only is preferable, since it captures the evolution of the network ``under its own pressure'', but it also seems especially effective in determining the so called {\it network backbone} \cite{WJS1}.

In contrast with our previous experiments in \cite{WJS}, \cite{WJS1}, \cite{WSJ2}, we employ here the {\it normalized flow}, in order to ensure that the metric does not contract (i.e. the weights do not converge to zero) and the network does not collapse to a node in the limit. This normalized flow for networks is defined, in analogy with that for surfaces (see, e.g. \cite{GY}), as 
%
%
%
%
\begin{eqnarray} \label{eq:NormalizedFormanRicciFlow}
\frac{\partial \omega(e)}{\partial t} = -\left( \rm{Ric} \left( \omega(e) \right)  - \rm{\overline{Ric}}\right) \cdot \omega(e)\,;
\end{eqnarray}
where $\rm{\overline{Ric}}$ denotes the mean Ricci curvature, and ${\rm Ric}$ stand for the graph Forman, full Forman or Haantjes Ricci curvature, according to the chosen curvature. Also, since in the setting of networks time is also assumed to evolve in discrete steps and each ``clock” (that is, time step)
has a length of 1, the Ricci flow takes the following form:
\begin{eqnarray} \label{eq:NormalizedFormanRicciFlow}
\tilde{\omega}(e) - \omega(e) = -\left( \rm{Ric} \left( \omega(e) \right)  - \rm{\overline{Ric}}\right) \cdot \omega(e)\, ; 
\end{eqnarray}
where  $\tilde{\omega}(e)$ denotes the new (updated) value of $\omega(e)$ (and $\omega(e)$ is the original 
i.e. given - one).

To emphasize its capabilities, we illustrate the flow on a  larger network, more precisely on the ``Windsurfers'' social network \cite{Kun}  -- see Figures 10-12. 
Note that a certain lower threshold for the edge weights needs to be predetermined, so ``noise'', that is to say edges with weights close to zero would not be included in the next iteration. In the experiments included here we have chosen to adopt a rather rather high threshold, such that at each step, at least $90\%$ of the edges appearing in the previous step would be preserved. 
Furthermore, in the computations using the graph Forman-Ricci curvature, the need to normalize curvature arose, in order to prevent abrupt jumps/collapses in curvature. However, no such normalization was needed for the Haantjes-Ricci curvature flow, a fact that seems to point to a relative advantage of this type of curvature, over the graph Forman one, as it seems to preclude fast collapse.

\begin{figure}[htb]
	\begin{center}
		\includegraphics[scale=0.575]{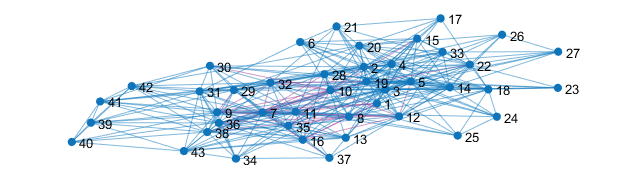}
		\includegraphics[scale=0.575]{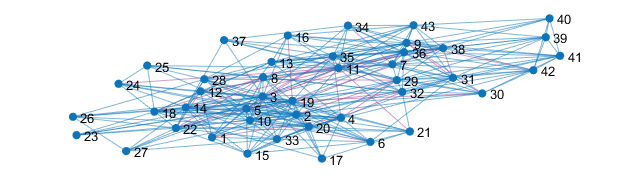}
		\includegraphics[scale=0.575]{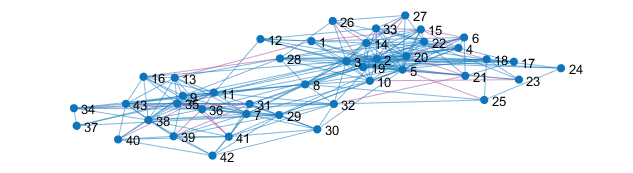}
		\includegraphics[scale=0.575]{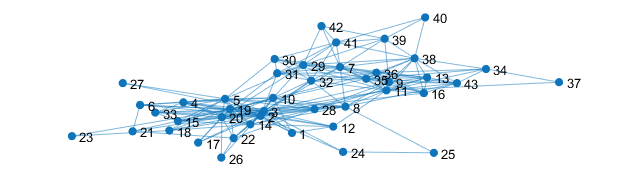}
		%
		%
	\end{center}
	\caption{The ``Windsurfers'' network and its evolution under the graph Forman-Ricci flow. In descending order, from above: The original network, and the network after 1, 3 and 5 iterations, respectively. }
\end{figure}

\begin{figure}[htb]
	\begin{center}
		\includegraphics[scale=0.575]{Windsurfers.png}
		\includegraphics[scale=0.45]{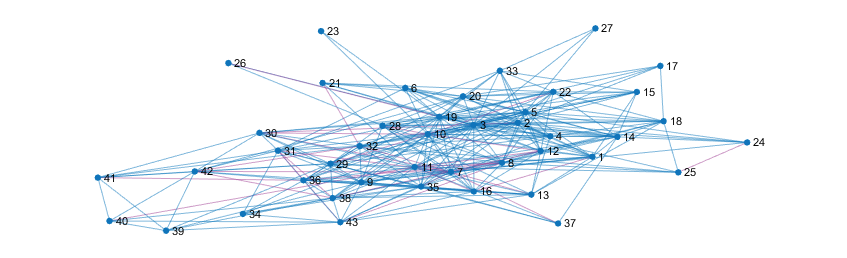}
		\includegraphics[scale=0.45]{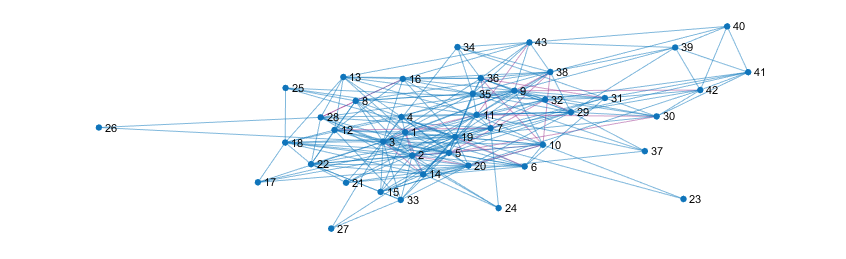}
		\includegraphics[scale=0.45]{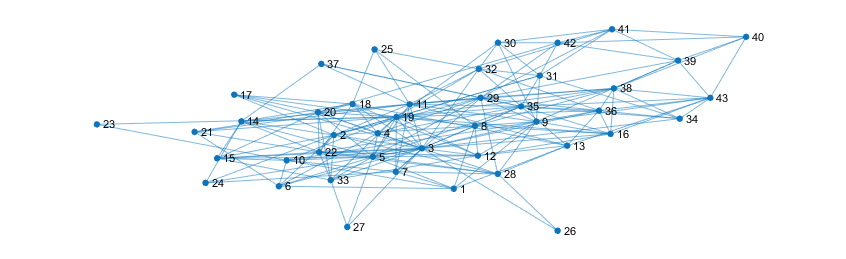}
		%
		%
	\end{center}
	\caption{The ``Windsurfers'' network and its evolution under the full Forman-Ricci flow. In descending order, from above: The original network, and the network after 1, 3 and 5 iterations, respectively. Note that, at least in this example, the Haantjes-Ricci flow is decreasing the size of the evolving network faster than the graph Ricci curvature.}
\end{figure}

\begin{figure}[htb]
	\begin{center}
		\includegraphics[scale=0.575]{Windsurfers.png}
		\includegraphics[scale=0.45]{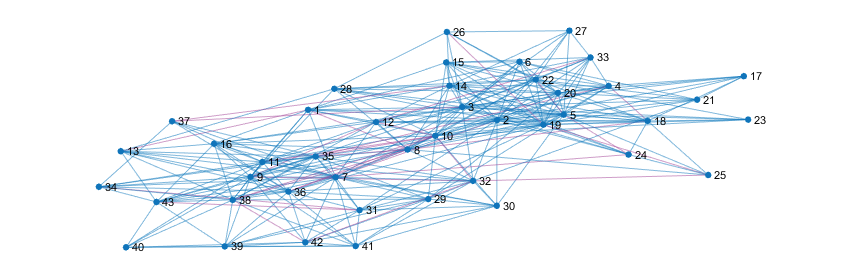}
		\includegraphics[scale=0.45]{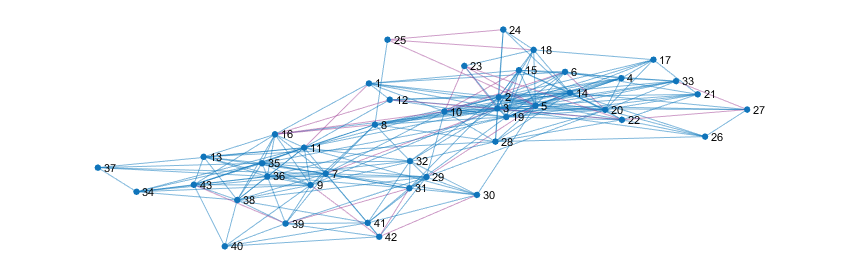}
		\includegraphics[scale=0.45]{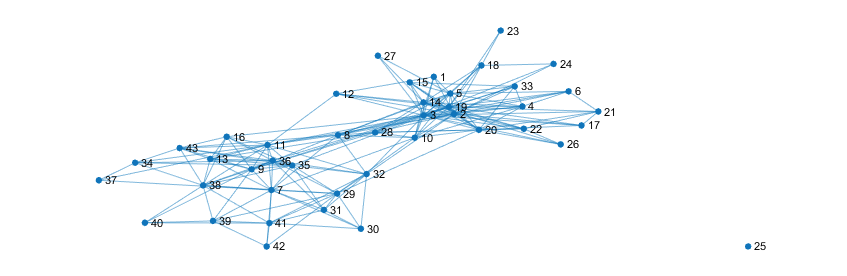}
		%
		%
	\end{center}
	\caption{The ``Windsurfers'' network and its evolution under the Haantjes-Ricci flow. In descending order, from above: The original network, and the network after 1, 3 and 5 iterations, respectively. Note that, at least in this example, the Haantjes-Ricci flow is decreasing the size of the evolving network faster than the graph Ricci curvature.}
\end{figure}

Note that, even only  a very small number of iterations were considered, Ricci curvature evolves, as expected from the theory in the smooth and combinatorial (circle-packing based) cases, to a constant. This phenomenon is most manifest in the Graph Forman-Ricci flow, and the least evident for the Full Forman-Ricci flow. Beyond the specific convergence speeds of each of the flows, this phenomenon occurs probably because, as already mentioned above, only triangular and quadrangular faces were considered, a fact which greatly simplifies and speeds-up computations, but also changes drastically the topology of the evolving 2-complex.

\begin{figure}[htb]
	\begin{center}
		\includegraphics[scale=0.43]{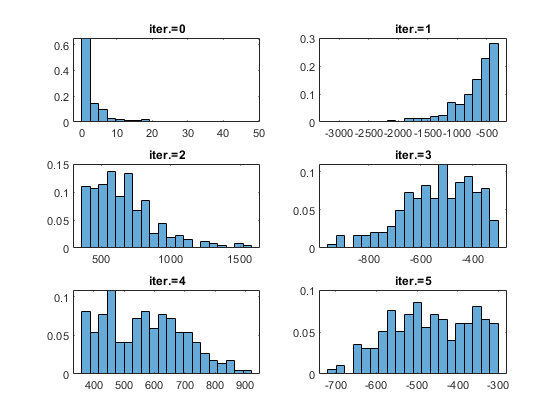}
		\includegraphics[scale=0.43]{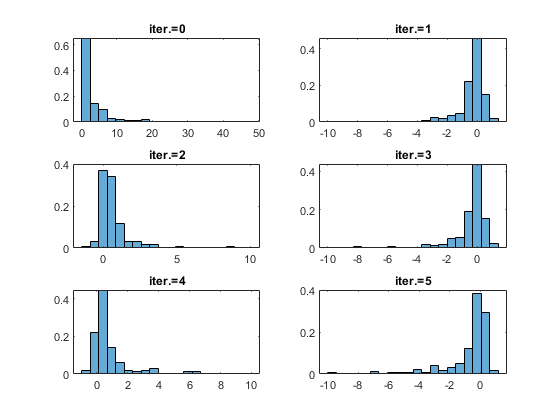}
		\includegraphics[scale=0.43]{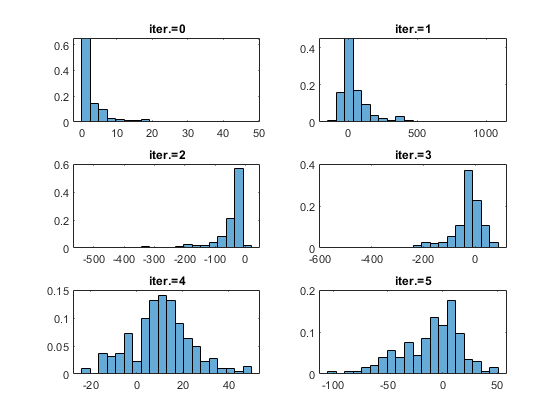}
		%
		%
	\end{center}
	\caption{The evolution of curvature of the ``Windsurfers'' network  under each of the studied Ricci flows. In descending order, from above: The Graph Forman-Ricci curvature flow, Full Forman-Ricci curvature flow and the Haantes-Ricci flow. Observe the extremely fast convergence towards a constant under the Graph Forman-Ricci curvature flow.}
\end{figure}

\begin{rem}
	As noted in \cite{WJS} a scalar (Yamabe) flow is also possible. However, here we concentrated on the Ricci curvature, that is to say edge-centric approach. 
\end{rem}



\section{Coarse Embedding  and an Application to SVM} \label{sec:Kernels}

A common theme in Geometry and in many of its applications (in some cases still in a non-deliberate manner) is the embedding, with least distortion, of a given structure/space, in a familiar, larger ambient space. For instance, in classical Geometry and Topology, this ambient space is usually  $\mathbb{R}^n$, for some $n$ large enough, whereas any metric space is isometrically embeddable in $l^\infty$. Also, more recently, embeddings of networks in the Hyperbolic Plane $\mathbb{H}^2$ or Space $\mathbb{H}^3$ have become quite common (see, e.g. \cite{Boguna2020} and the references therein).

It is therefore most natural to ask, not only if a {\it coarse embedding} (viewed as a metric measure space) of a weighted graph exists, but also whether there exists an ``automatic'' procedure to achieve such an embedding. By a coarse embedding of a metric space $(X,d)$ into another metric space $(Y,\rho)$, we mean a map $i: X \rightarrow Y$, such that there exist increasing, unbounded functions $\eta_1,\eta_2:\mathbb{R} \rightarrow \mathbb{R}$, such that 
\begin{equation} \label{eq:CoarseEmbed}
\eta_1(d(x_1,x_2)) \leq \rho(i(x_1),i(x_2)) \leq \eta_2(d(x_1,x_2))\,,
\end{equation}
for any $x_1,x_2 \in X$.

The ``automatic'' embedding method we alluded to above is one that is canonical in SVM techniques, and classical in Fourier Analysis and its various applications, in particular in Signal and Image Processing, namely that of {\it reproducing kernels}. Recall the following

\begin{defn}
Given a set $X$, {\it symmetric kernel} is a symmetric function  $k:X \times X  \rightarrow \mathbb{R}$, i.e. $k(x,y) = k(y,x)$, for any $x,y \in X$.

A kernel $k$ is said to have 
\begin{enumerate}
	\item {\it positive type} if the matrix $K_m = \{k(x_i,x_j)\}_{i,j=1}^m$ is positive semidefinite for all $m \in \mathbb{N}$;
	
	and
	
	\item {\it negative type} if the matrix $K_m = \{k(x_i,x_j)\}_{i,j=1}^m$ is negative semidefinite for all $m \in \mathbb{N}$;
\end{enumerate}
		
\end{defn}

Positive and negative type kernels are interrelated by the following classical result

\begin{prop}[Schoenberg's Lemma]
Let $k$ be a symmetric kernel on set $X$. Then the following statements are equivalent:
\begin{enumerate}
	\item The kernel $k$ is of negative type;
	
	\item The kernel $\kappa = {\rm exp}(-tk)$ is of positive type, for each $t > 0$.
\end{enumerate}
\end{prop}
(For a proof see, e.g. \cite{Roe}.)

From Schoenberg's Lemma and the fact that $k^\alpha$ can be written as 
\[
k^\alpha(x,y) = C\int_0^\infty(1 - e^{-t\kappa(x,y)})t^{-\alpha/2 - 1}dt\,
\]
where $C = C_\alpha$ is a positive constant, 
it follows (cf. \cite{Roe}) that following corollary hods: 

\begin{cor}
	Let $k \geq 0$ be a kernel of negative type on $X$. Then $\kappa = \kappa^\alpha$ is also of negative type, for any $0 < \alpha < 1$.
\end{cor}

\begin{rem}
	An (effectively coarse) embedding method for networks as well as higher dimensional spaces was proposed in \cite{Sa17}. While the approach suggested therein applies to more general spaces than networks and, the embedding space is the familiar $\mathbb{R}^n$, it still is less than intuitive, since it makes appeal to a family of quasi-metrics. While this approach allows for the the network to be studied at many scales, it is also more complicated then the one adopted here (and based on the path degree metric). Furthermore, the method introduced herein has also the advantage of coming in conjunction with a reproducing kernel, that allows for the automation required in SVM related applications. 
	Therefore, the corollary above shows that a way of studying networks at many scales, akin to the one in \cite{Sa17}, exists also in the coarse embedding (kernel) approach.
\end{rem}

Since the operators $\Box_1$ and $B_1$ are symmetric as functions of the end nodes $u,v$ of an edge $e = (u,v)$, they define (in analogy to the classical Laplacian) reproducing kernels $k_{\Box}, k_B$ on any given network. Moreover, since by its very construction/definition $B_p$ is a positive semidefinite, it follows that $k_p$ is also positive semidefinite (i.e. a ``classical'' reproducing kernel). By a direct application of a  classical result (for a proof, see, e.g. \cite{Roe}, Theorem 11.15), we get 

\begin{thm}
	Let $k$ be a symmetric kernel on $X$. Then,
	
	\begin{enumerate}
		\item If $k$ is of positive type, then there exists a map $\varPhi:X \rightarrow \mathcal{H}$, where $\mathcal{H}$ is a real Hilbert space, such that $k(x,y) = <\varPhi(x),\varPhi(y)>$, for any $x,y \in X$.
		
		\item If $k$ is of negative type and, furthermore, $k(x,x) = 0$ for any $x \in X$, then there exists a map $\varPhi:X \rightarrow \mathcal{H}$, where $\mathcal{H}$ is a real Hilbert space, such that $k(x,y) = ||\varPhi(x) - \varPhi(y)||^2$, for any $x,y \in X$.
	\end{enumerate}
\end{thm}
 %

Moreover, ${\bf F}$ and ${\rm Ric}_F$ are both symmetric functions (again viewed as acting on pairs of nodes defining edges), thus kernels $k_{\bf F}$ and $k_{{\rm Ric}_F}$ can be defined.
However, neither version of Forman's Ricci curvature  is positive, therefore a mapping into a real Hilbert space, as for the Laplacians, is not possible for these curvature-based kernels. In fact, ${\bf F}$ is {\em not} negative only if both its end nodes have degree 2, that is only the degenerate case of cycles and graph ``spurious'' vertices of degree 2, i.e graph subdivisions {\it homeomorphic}
to a ``good'' graph can have edges of non-negative curvature. It can, therefore, be surmised that proper networks have pure Forman curvature  ${\bf F} < 0$. (Note that this does not hold for ${\rm Ric}_F$.)

Unfortunately, part (2) of the theorem above (the one regarding negative type operators) does not hold, since it requires that $k(x,x) = 0$, for all $x \in X$, which, of course, does not hold neither for $\Box_1$, nor for $B_1$. However, one can still map the kernels $k_{\Box}, k_B$ to a Hilbert space precisely as above, after modifying them into fitting related positive operators, by putting  $k^*_{\Box} = e^{-k_{\Box}}; \:\: k^*_B = e^{-k_B}$.\footnote{In fact, for every $t > 0$, the kernels $k^{*,t}_{\Box}(x,y) = e^{-tk_{\Box}(x,y)};$,  $k^{*,^t}_B(x,y) = e^{-tk_B(x,y)}$ are of positive type (see \cite{Roe}, Proposition 11.12.}

Moreover (and perhaps more important in our context) there exists a coarse embedding of any (finite) network into a (real) Hilbert space. This follows from the fact that, for instance, $e^{-k_F}$ is a positive kernel, and from the fact that 
the kernel $k_F$, is {\it effective} i.e. the edges $\{e = (x,y)\,|\, |k(e)| < K\}$, $K > 0$ generate the {\it coarse structure} of the network (see \cite{Roe} for technical details), which, given in the case at hand, of finite networks, can naturally to be taken as the  {\it discrete coarse structure} generated by sets that contain only a finite number of points (nodes) off the diagonal. Therefore we can apply the following result:

\begin{thm}[\cite{Roe}, Theorem 11.16.]
	Let $X$ be a coarse space. Then the following statements are equivalent:
	\begin{enumerate}
		\item $X$ can be coarsely embedded into a Hilbert space.
		
		\item There exists an effective negative type kernel on $X$. 
	\end{enumerate}
\end{thm}

In fact, the observation above can be strengthened in more than one way: On the one hand, the result can be applied to any of the kernels above, not only the one considered above. Furthermore, due to finiteness, one can relax the effectiveness condition by dispensing with the absolute value in the defining inequality. On the other hand, the Forman-Ricci curvature kernels are effective even if one considers $\varepsilon$-nets obtained from the geometric sampling of non-compact Riemannian manifolds with Ricci curvature bounded from below, given that the sampling procedure is essentially the same as for the compact ones (see \cite{Kanai} for the classical case and \cite{Sa11-1} for the generalization to metric measure spaces). 
The only difference between the compact case and the one at hand is that one has to relax somewhat the coarseness isometry definition and the resulting $\varepsilon$-net, endowed with the combinatorial metric (i.e. with edge lengths $\equiv 1$), will be only {\it roughly} isometric to the given manifold, where rough isometry is defined as follows:

\begin{defn}[Rough isometry]
	Let $(X,d)$ and $(Y,\delta)$ be two metric spaces, and let $f:X \rightarrow Y$ (not necessarily continuous).
	$f$ is called a {\em rough isometry} iff
	
	\begin{enumerate}
		\item There exist $a \geq 1$ and $b > 0$, such that
		\[\frac{1}{a}d(x_1,x_2) - b \leq \delta(f(x_1),f(x_2)) \leq ad(x_1,x_2) + b\,,\]
		%
		\item there exists $\varepsilon_1 > 0$ such that
		\[\bigcup_{x \in X}{B(f(x),\varepsilon_1)}= Y;\]
		(that is $f$ is $\varepsilon_1$-{\it full}.)
	\end{enumerate}
	
\end{defn}


It follows that all types of kernels based on Forman's discretization of the Bochner-Weitzenb\"{o}ck may be applied in the variety of SVM problems where kernels are usually employed, such as clustering and classification. 
In particular, one can compute the so called {\it kernel distance} \cite{JKPV}, \cite{PV}: 

\begin{defn}
Given a {\it similarity function} ({\it kernel}) $K:\mathbf{R}^d \times \mathbf{R}^d \rightarrow \mathbf{R}$, such that $K(p,p) = 1$, for any $p \in \mathbb{R}^d$, and given sets $P,Q \subset \mathbf{R}^d$, the {\it kernel distance} $D_K(P,Q)$ is defined as 
\begin{equation}
D_K(P,Q) = \sqrt{\sum_{p \in P}\sum_{p' \in P}K(p,p') - 2\sum_{p \in P}\sum_{q \in Q}K(p,q) + \sum_{q \in Q}\sum_{q' \in Q}K(q,q')}\,.
\end{equation}
\end{defn}

\begin{rem}
	$D_K(P,Q)$ is not a metric if $K$ is not positive definite. 
\end{rem}

Note that in the special case $P = \{p\}, Q = \{q\}$, we obtain $D^2_K(P,Q) = 2(1 - K(p,q))$, thus $1 - K(p,q)$ can be viewed as a proper squared distance, since $1 - K(p,p) = 0$.

\begin{rem}
	The quantity $\kappa(P,Q) = \sum_{p \in P}\sum_{q \in Q}K(p,q)$ is called the {\it cross-similarity} (of $P$ and $Q$), thus $D_K(P,Q)$ can be expressed, more simply, in terms of the cross similarity, rather then in means of the kernel as $\kappa(P,P) - 2\kappa(P,Q) + \kappa(Q,Q)$.
\end{rem}

A first application of the kernel distance is in the visualization of kernel spaces and data in general \cite{SMcC}, \cite{JKPV}. The embedding method employed here is the so called {\it mutidimensional scaling} ({\it MDS}) one (see, e.g. \cite{BG}). 
More precisely, given $N$ points $p_1,\ldots,p_N$, we can approximate their position in a kernel space $\mathcal{H} \subset \mathbb{R}^d$ by $y_1,\ldots,y_N \in \mathbb{R}^d$, by minimizing the {\it cost-function}
\begin{equation}
J = \sum_{i = 1}^N\sum_{j > i}^N(||y_i - y_j|| - D_{ij})^2\,;
\end{equation}
where $D_{ij} = D_K(p_i,p_j)$.

We illustrate this application below, using the Graph Forman-Ricci curvature, on two real-life networks, namely the simple, illustrative ``Kangaroo'' network, and on the larger ``Les Mis\'{e}rables'' one. As the second example proves, this visualization method is quite efficient at distinguishing various clusters in a large network, due to the high degree of separation that higher dimensional spaces afford. In these examples, to numerically determine the minimum of the {\it cost-function} $J$ we made appeal to the classical {\it Douglas-Rachford algorithm} \cite{DR}. We have illustrated these examples in Figures 13 and 14 below.

\begin{figure}
	\includegraphics[scale=0.17]{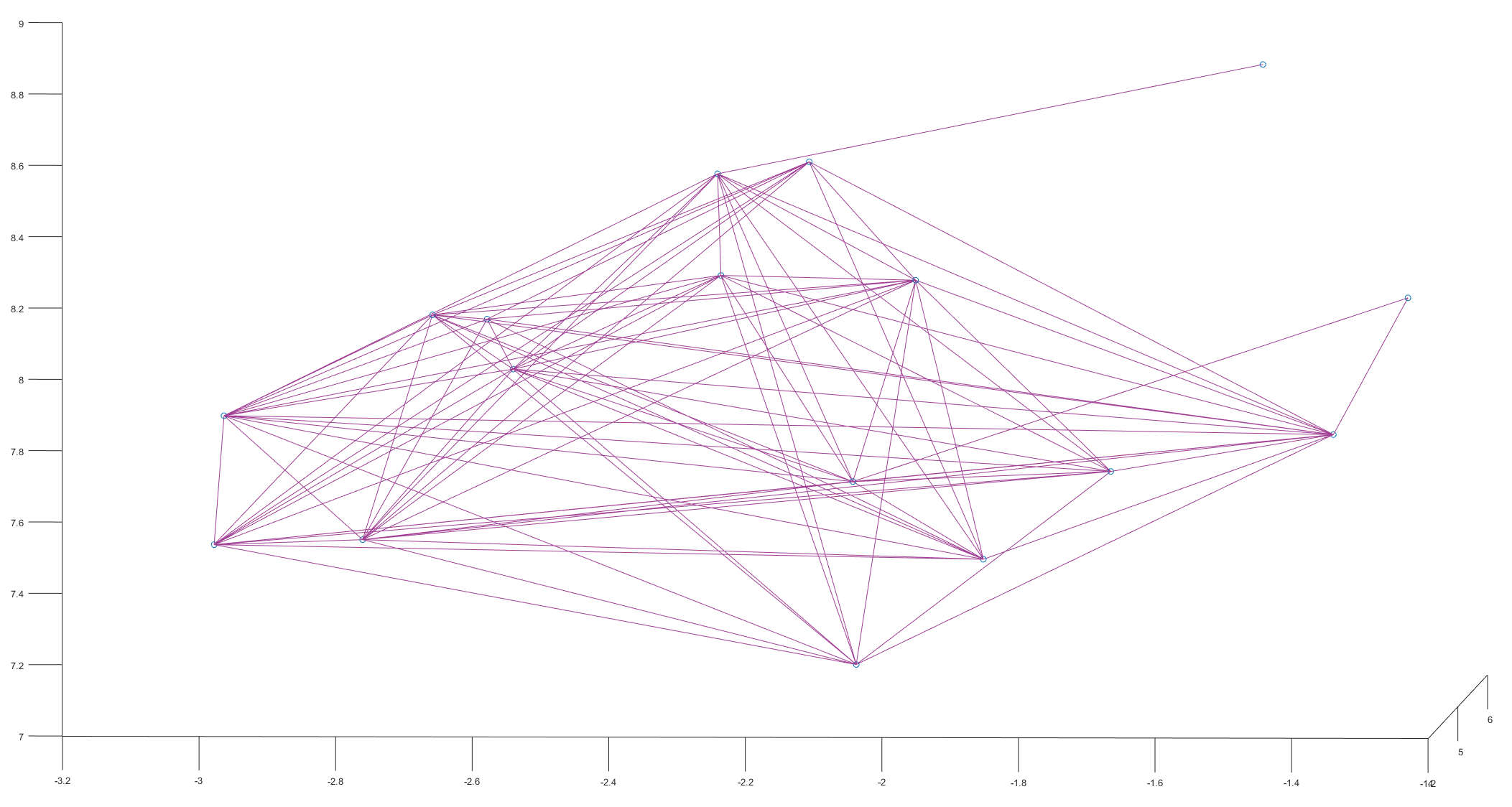}
	\includegraphics[scale=0.17]{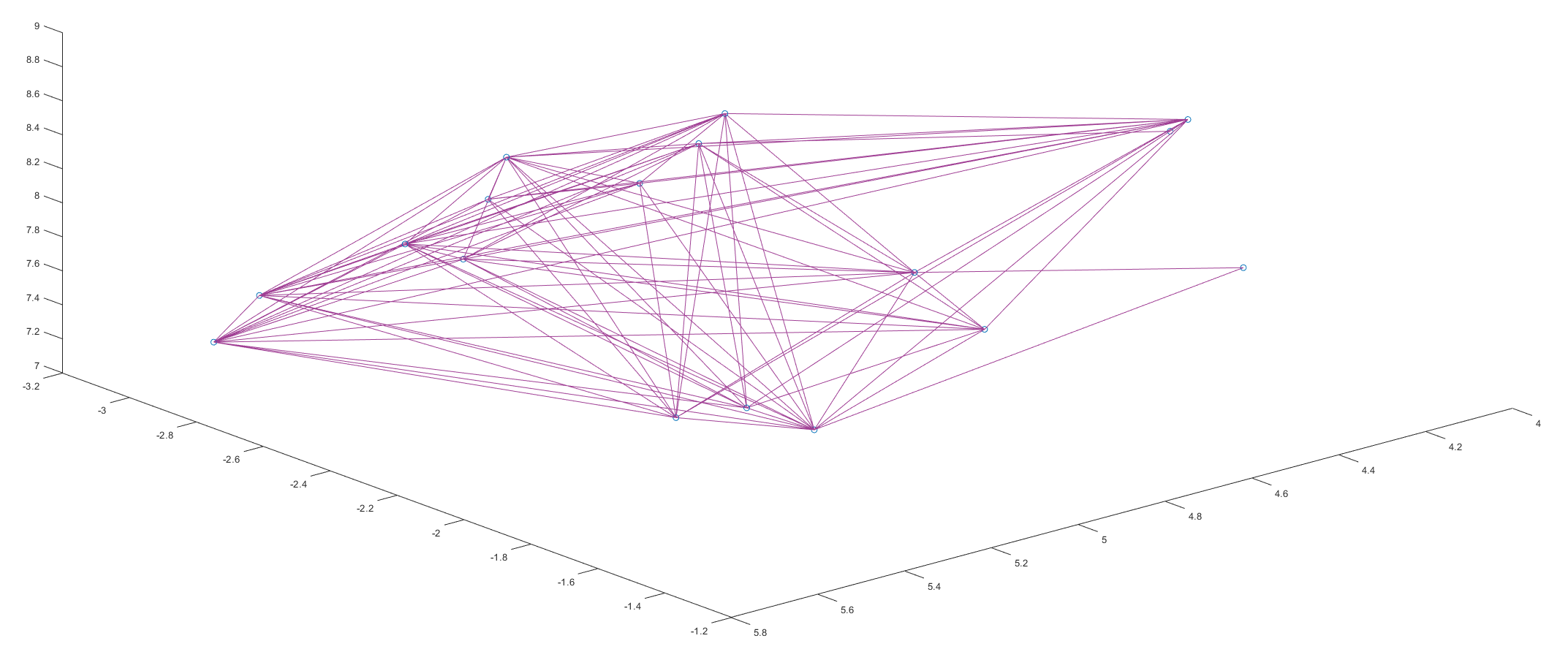}
	\caption{Two views of the ``Kangaroo'' Graph Forman-Ricci curvature derived, kernel-based,  embedding in $\mathbb{R}^3$ which reveal that the network essentially consist from one large cluster. The the attained minimum for the cost function in this case is 
		$J = 23.841747$.}
\end{figure}

\begin{figure}
	\includegraphics[scale=0.17]{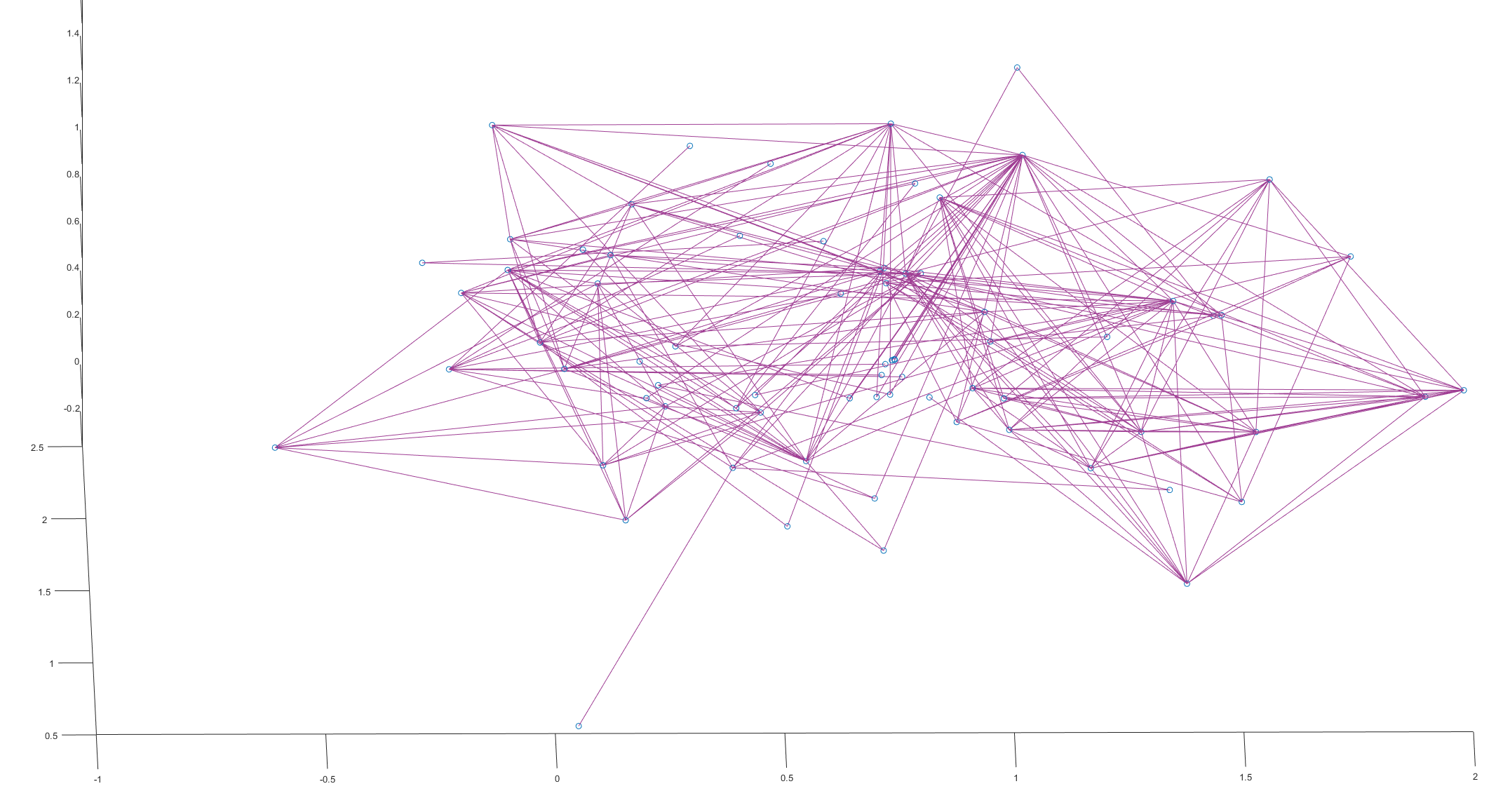}
	\includegraphics[scale=0.17]{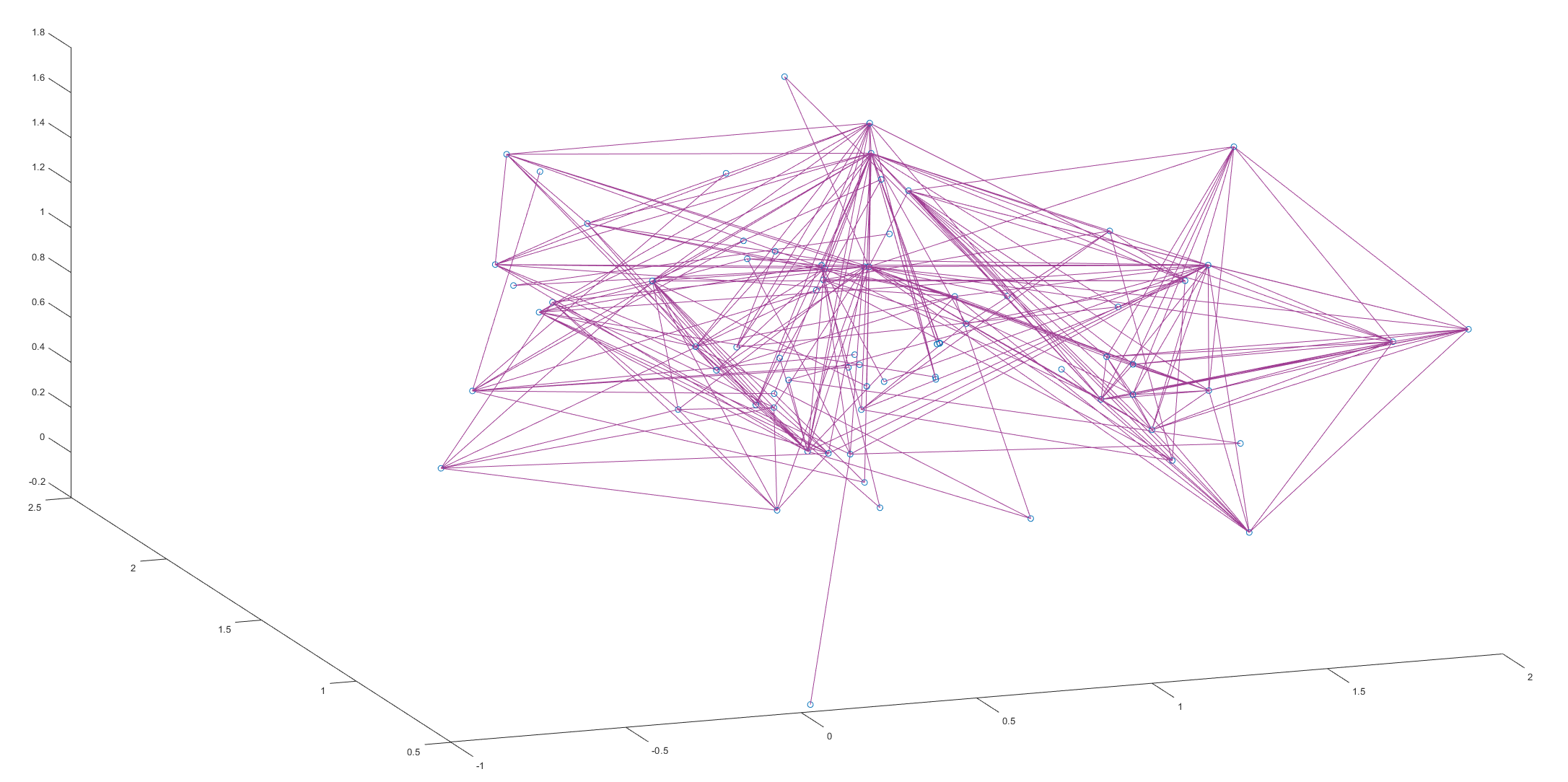}
	\caption{Two views of the Graph Forman-Ricci curvature derived, kernel-based embedding of the ``Les Mis\'{e}rables'' characters network. Here the attained minimum for the cost function is $J = 42.264625$. Note that the kernel-based embedding into $\mathbb{R}^3$ clearly 
    distinguishes the clusters around each of the main characters in Hugo's novel, connected by the relations between them.}
\end{figure}

Furthermore, we have implemented the $\square_1$ based embedding. Perhaps contrary to the common wisdom, and thus somewhat unexpected, the Graph Forman-Ricci curvature renders a far better separation into clusters than the considered Laplacian. -- See Figure 16. 

\begin{figure}
	\includegraphics[scale=0.17]{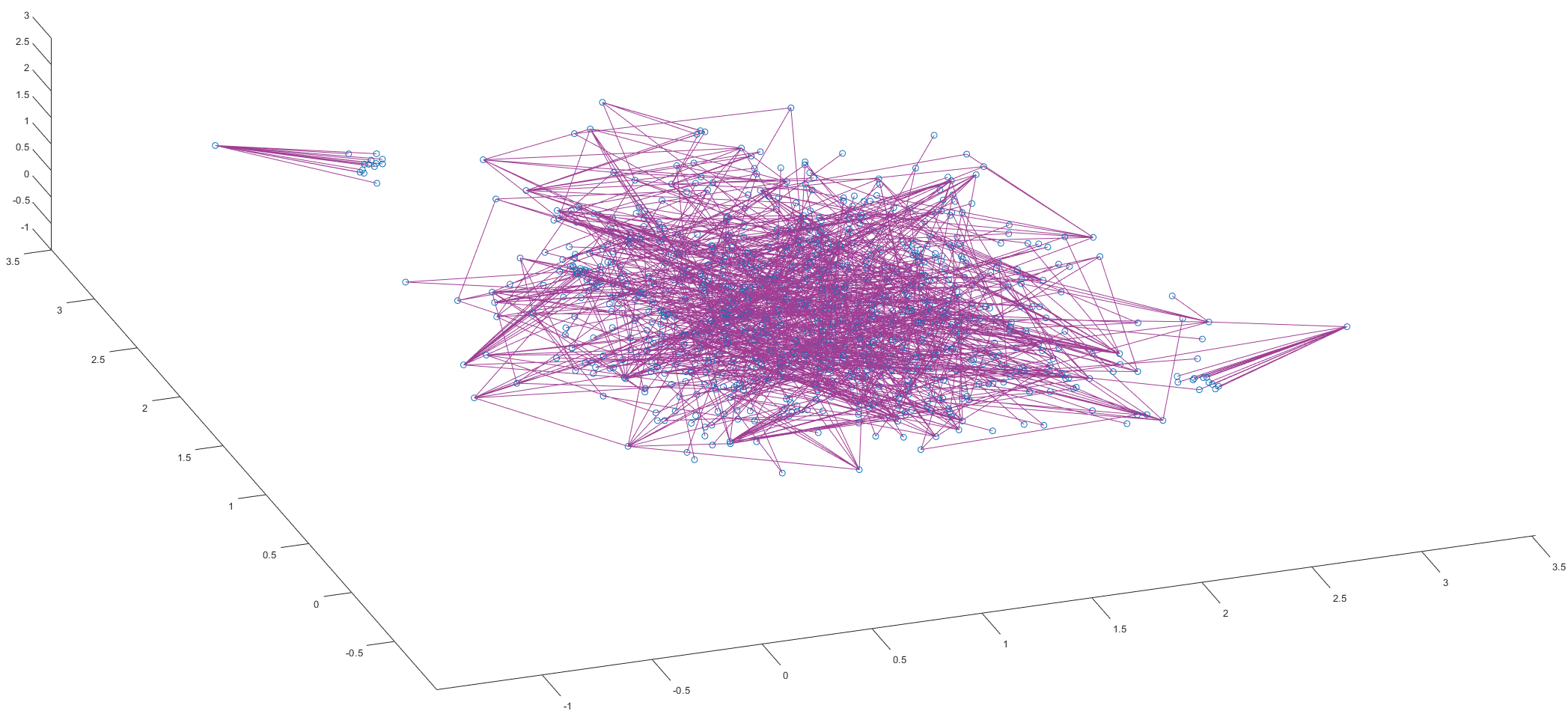}
	\includegraphics[scale=0.17]{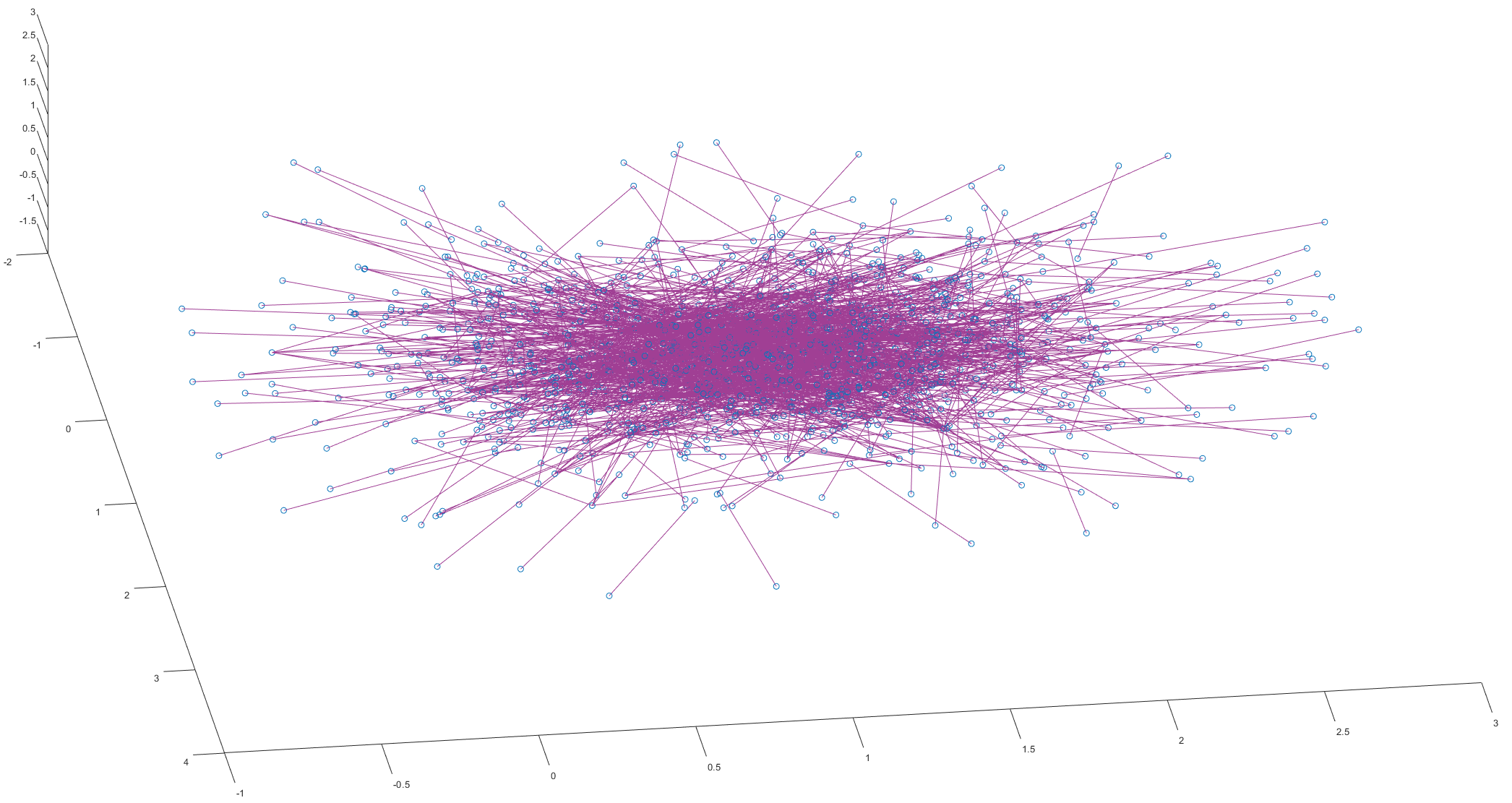}
	\caption{Graph Forman-Ricci (above) vs. $\square_1$ (below) kernel-based embedding of the ``C. Elegans'' network \cite{Cho++}. Note the far better separation properties of the curvature based embedding. This ability has as a penalty a higher minimum of the cost function $J$, namely 404.342321 for the curvature-based kernel, as compared to 283.640003 for the Laplacian-based kernel.}
\end{figure}

The kernel embedding method not only is useful in the visualization of networks, it also (and perhaps more importantly) allows us to embed large complex networks into the familiar Euclidean space. Consequently one can employ well established and intuitive sampling methods of data points and sets in the usual ambient space.



\section{From Manifold Sampling to Network Sampling}\label{section:manif->ntwks}

Given that the discretizations of sampled manifolds are graphs (moreover, also named $\varepsilon$-nets), conducts one to pose the following natural question, namely weather there is a connection between the curvature bounds of the given manifold and the discrete curvature of the resulting graph, more specifically their Forman-Ricci curvature. The answer to this question is positive, as we shall demonstrate below. 

\begin{rem}
This fact , in conjunction with the result in \cite{Sa-Morse} on the connection between classical (smooth) curvature of a given Riemannian manifold and the combinatorial curvature of an approximating $PL$ manifold, significantly improves our understanding of the interrelations between the curvature of an underlying smooth manifold and its various discretizations. 
\end{rem}

The proof follows easily from the following essential properties of efficient packings that are used in the proof of the Grove and Peterson's result mentioned in Section \ref{section: background}:

Let $M^n$ be a an $n$-dimensional Riemannian manifold, 
such that ${\rm Ric}_M \geq (n-1)k$,
and let $D$ denote the
upper bound of the diameter of $M^n$. Then the following lemmas hold:

\begin{lem}[\cite{GP}, Lemma 3.2] \label{lem:GP1}
	There exists $n_1 = n_1(n,k,D)$, such that if
	$\{p_1,\ldots,p_{n_0}\}$ is a minimal $\varepsilon$-net on $M^n$, then $n_0
	\leq n_1$.
\end{lem}

\begin{lem}[\cite{GP}, Lemma 3.3] \label{lem:GP2}
	There exists $n_2 = n_2(n,k,D)$, such that for any $x \in M^n$,
	$\left|\{j \,|\, j = 1,\ldots,n_0\; {\rm and}\;
	\beta^n(x,\varepsilon) \cap \beta^n(p_j,\varepsilon) \neq
	\emptyset\}\right| \leq n_2$, for any minimal $\varepsilon$-net
	$\{p_1,\ldots,p_{n_0}\}$.
\end{lem}

\begin{lem}[\cite{GP}, Lemma 3.4] \label{lem:GP3}
	Let $M_1^n, M_2^n$, be manifolds having the same bounds $k =k_1 = k_2$
	and $D = D_1  = D_2$ (see above) and let $\{p_1,\ldots,p_{n_0}\}$ and
	$\{q_1,\ldots,q_{n_0}\}$ be minimal $\varepsilon$-nets with the same
	intersection pattern, on $M_1^n$, $M_2^n$, respectively. Then
	there exists a constant $n_3 = n_3(n,k,D,C)$, such that if
	$d(p_i,p_j) < C\cdot\varepsilon$, then $d(q_i,q_j) <
	n_3\cdot\varepsilon$. 
\end{lem}
%

In other words, by using curvature-based sampling, the resulting $\varepsilon$-nets (discretizations) have a number of properties, listed below in the same order as the implying lemmas, that are essential in the sampling task:

\begin{enumerate}
	\item There is an upper bound on the number of sampling points.
	
	\item The vertex degree is bounded from above.
	
	\item The distances between sampling points are bounded from above, that is to say the lengths of the edges of the discretization are. Here the distance is the intrinsic distance on the sampled manifold. However, in practice on can (and commonly does, in Graphics and Imaging applications) replace the intrinsic distances with the Euclidean ones in a piecewise-flat approximation obtained via an isometric embedding in a higher dimensional Euclidean space. (See \cite{SAZ} and the references therein, as well as \cite{SAZ1} for concrete bounds for the relation between the intrinsic and approximated distances.) 
\end{enumerate}

\begin{rem}
	The lemmas above also hold in the more general case of metric measure spaces with a lower bound on the generalized Ricci curvature ${\rm CD}(K,N)$ \cite{Sa11-1}. While we do not cite the precise  statements here, in order to avoid unnecessary technical complications, the reader should recall that these, as well all the results below also hold for metric measure spaces with generalized Ricci curvature bounded from below. 
\end{rem} 

Let us examine the implications to the networks' sampling setting of each of the properties above:

The implications of Property (1) are the clearest from the Signal and Image Processing viewpoint: Given a prescribed precision of approximation ($\varepsilon$), there there exists an upper bound on the number of required sampling points (that depends on dimension, curvature bound and diameter). 

Property (2) is the one that allows for the basic connection with the Forman-Ricci curvature of the $\varepsilon$-separated net. Indeed, since for networks endowed with combinatorial weights (that is to say equal to 1), the formula for the simple graph Forman-Ricci curvature has the following alluring form:
\begin{equation}
\mathbf{F}(e) = 4 - {\rm deg}(v_1)  - {\rm deg}(v_2); 
\end{equation}
where $v_1, v_2$ are the end nodes of the edge $e$, 
property (2) translates to the following

\begin{lem} \label{lem:Ricci->gF-combi}
	Let $G(\mathcal{N})$ be a $\varepsilon$-separated net of a bounded manifold $M$.
	Then the graph Forman-Ricci curvature of  $G(\mathcal{N})$ is bounded from below; more precisely 
	\begin{equation}
     \mathbf{F}(e) \geq K_1; {\rm for\; any}\; e\;  {\rm edge\; of}\; G(\mathcal{N})\,;
	\end{equation}
	where $K_1 = 1 - 2n_2$ and $n_2 = n_2(n,k,D)$ is given by Lemma \ref{lem:GP2} above.
\end{lem}

In fact, the result above extends, in view to a result of Kanai \cite{Kanai}, to all discretizations of bounded curvature of a given manifold $M$. More precisely, we have that

\begin{thm}[\cite{Kanai}, Lemma 2.5]
Let $M^n$ be a complete Riemannian manifold, such that ${\rm Ric}_M \geq (n-1)k$, and let $G(\mathcal{N})$ be a discretization of $M^n$. Then $(M^n,d)$ and $(G,\mathbf{d})$,  where $d$ is the Riemannian metric and $\mathbf{d}$ is the combinatorial metric, are roughly isometric.
\end{thm}

Therefore, it follows that we can also state 

\begin{lem}  \label{lem:Ricci->gF-combi1}
	Let $M^n$ be a complete Riemannian manifold, such that ${\rm Ric}_M \geq (n-1)k$, and let $G(\mathcal{N})$ be a discretization of $M^n$, endowed with the combinatorial metric. Then the graph Forman-Ricci curvature of  $G(\mathcal{N})$ is bounded from below; more precisely 
	\begin{equation}
	\mathbf{F}(e) \geq K_2; {\rm for\; any}\; e\;  {\rm edge\; of}\; G(\mathcal{N})\,;
	\end{equation}
	where $K_2 = 1 - 2k_1$, and $k_1$ is the upper bound on the degrees of the vertices of $G(\mathcal{N})$. 
	
\end{lem}

\begin{rem}
	In view of the results in \cite{Sa11-1}, Lemmas  \ref{lem:Ricci->gF-combi} and \ref{lem:Ricci->gF-combi1} above both  readily adapt to metric measure spaces with generalized Ricci curvature bounded from below.
\end{rem}

Moreover, the lower bound on the Ricci curvature translates to a similar bound on the graph Forman-Ricci curvature of a discretization, even if this is not endowed with the discrete metric, but rather with the (arguably more natural) induced metric. (By this we mean that the length of an edge equals the distance on $M^n$ between its edge points.)

\begin{prop}
	Let $M^n$ be a Riemannian manifold with Ricci curvature bounded from below, i.e. ${\rm Ric}_M \geq (n-1)k$ and let $G(\mathcal{N})$ be a discretization with bounded geometry, having edge weights equal to the Riemannian distance between the end points. Then the graph Forman-Ricci curvature is also bounded.
\end{prop}

\begin{proof}
	First of all, let us note that the node (vertex) weights also need to be prescribed. There are a number of natural possibilities: The combinatorial weight 1, the degree ${\rm deg}(v)$ and the {\it weighted degree} ${\rm deg}_w(v) = \frac{1}{{\rm deg}(v)}\sum_{e \sim v} w(e)$. The first one is less fitting in this geometric context, while the last one introduces the edge weights again, thus is repetitive in its modeling capacity. We settle, therefore, for the second option, i.e. $w(v) = {\rm deg}(v)$, even though, for the purpose of this proof, no specific choice is necessary.

	Recall that by Formula (\ref{FormanRicciEdge}), the graph Forman-Ricci curvature is 
	\[
	\mathbf{F}(e) = w(e) \left( \frac{w(v_1)}{w(e)} +  \frac{w(v_2)}{w(e)}  - \sum_{e(v_1)\ \sim\ e,\ e(v_2)\ \sim\ e} \left[\frac{w(v_1)}{\sqrt{w(e) w(e(v_1))}} + \frac{w(v_2)}{\sqrt{w(e) w(e(v_2))}} \right] \right)\,;
	\]
	Therefore, $\mathbf{F}(e)$ is bounded iff the edge weights are bounded away from zero (independently of the specific choice of node weights). For bounded manifolds this fact follows immediately from the finiteness Property 1. (In fact, Property 3 gives also an upper bounds for the weights.) 
	For non-compact manifolds with bounded curvature, it follows from Kanai's Theorem that there exists $C_1,C_2 > 0$ such that $C_1 \leq w(e) \leq C_2$, thus the desired property also holds in this case. 
	
\end{proof}

\begin{rem}
	Again, by applying the results in \cite{Sa11-1}, the proposition above extends to metric measure spaces with generalized Ricci curvature bounded from above as well. 
\end{rem}

However, Property 3 does not suffice to assure a similar connection between the given upper bound on the curvature of a Riemannian manifold $M$ and one on the full Forman curvature of a discretization $G(\mathcal{N})$ of $M$. 
Indeed, in computing the full Forman curvature, one has to take into account not just the edge weights (lengths), but also the weights of the 2-faces, which in this case should naturally be taken as the areas of the corresponding $PL$ (or, rather, piecewise flat approximation). (To this end, we can assume, of course, that $M$ is isometrically embedded in some $\mathbf{R}^N$, for $N$ sufficiently large.) Recall also that in the sampling process one produces, in fact, simplicial complexes (triangulations), thus 2-face areas are readily computable from the edge lengths. Unfortunately, isometry does not imply non-collapse of area (for some extreme and important examples, see \cite{BZ}). Therefore, in order to obtained the desired result, we have to ensure that such collapse does not occur. This requirement is fulfilled if the given triangulation induce by the $\varepsilon$-net is {\it thick} (or {\it fat}); where thickness is defined as follows

\begin{defn} Let $\tau \subset \mathbb{R}^n$ ; $0 \leq k \leq n$ be a $k$-dimensional simplex.
	The {\it thickness}  $\varphi$ of $\tau$ is defined as being:
	\begin{equation}
	\varphi = \varphi(\tau) = \hspace{-0.3cm}\inf_{\hspace{0.4cm}\sigma
		\leqslant \tau
		\raisebox{-0.25cm}{\hspace{-0.9cm}\mbox{\scriptsize${\rm dim}\,\sigma=j$}}}\!\!\frac{\rm Vol_j(\sigma)}{\rm diam^{j}\,\sigma}\;.
	\end{equation}
	The infimum is taken over all the faces of $\tau$, $\sigma \leqslant \tau$,
	and ${\rm Vol}_{j}(\sigma)$ and ${\rm diam}\,\sigma$ stand for the Euclidian
	$j$-volume and the diameter of $\sigma$ respectively. (If
	${\rm dim}\,\sigma = 0$, then ${\rm Vol}_{j}(\sigma) = 1$, by convention.)
	A simplex $\tau$ is $\varphi_0${\it-thick}, for some $\varphi_0 > 0$,
	if $\varphi(\tau) \geq \varphi_0$. A triangulation (of a submanifold
	of $\mathbb{R}^n$) $\mathcal{T} = \{ \sigma_i \}_{i\in \bf I}$ is
	$\varphi_0${\it-thick} if all its simplices are $\varphi_0$-thick. A
	triangulation $\mathcal{T} = \{ \sigma_i \}_{i\in \bf I }$ is {\it
		thick} if there exists $\varphi_0 > 0$ such that all its
	simplices are $\varphi_0${\it-thick}.
\end{defn}

Since any Riemannian manifold, satisfying only mild topological finiteness conditions (with or without boundary) admits a fat triangulation -- see e.g. \cite{Sa11-1} and the references therein, it follows that any $\varepsilon$-net can be improved to render a thick triangulation. 
Furthermore, during the thickening process the edge lengths (weights) are modified only slightly, given that one uses to this end only $\varepsilon$-moves (see \cite{mun}, \cite{cms} for technical details).
Given the facts above and recalling that, by Formula (\ref{eq:Forman-2d}), the full Forman-Ricci curvature is expressed by
\[
\hspace*{-2.65cm}{\rm Ric}_{\rm F} (e) = \omega (e) \left[ \left( \sum_{e \sim f} \frac{\omega(e)}{\omega (f)}+\sum_{v \sim e} \frac{\omega (v)}{\omega (e)}	\right) \right. 
\]
\[\hspace*{2.65cm}
- \left. \sum_{\hat{e} \parallel e} \left| \sum_{\hat{e},e \sim f} \frac{\sqrt{\omega (e) \cdot \omega (\hat{e})}}{\omega (f)} - \sum_{v 	\sim e, v \sim \hat{e}} \frac{\omega (v)}{\sqrt{\omega(e) \cdot \omega(\hat{e})}} \right| \right] \; ;
\]
 we have in fact proven the following theorem:

\begin{thm}
 Any  Riemannian manifold with Ricci curvature bounded from below admits a discretization with bounded full Forman-Ricci curvature.
\end{thm}

\begin{rem}
This result can also be extended to manifolds satisfying a generalized Ricci curvature bound by applying \cite{Sa11-1}, Proposition 4.7.
\end{rem}

Before concluding this section, let us note that, in the process of simplex thickening,  a (finite) number of subdivisions is required. In consequence, the resulting triangulation and the discretization associated to it will have a larger number of vertices, with degrees that do not depend simply on dimension, diameter and curvature. Therefore, the need for a simplification of the resulting network might arise in its turn. To this end, a Forman-Ricci curvature (of either kind) sampling would be applicable. Since the degrees of the additional vertices are a function depending solely of the dimension of the subdivided simplex, it is therefore quite probable that, at least when considering the discrete metric and applying graph Forman-Ricci curvature, the remaining nodes (and edges) would mostly still be the ones of the initial triangulation. 


\section{Conclusion and Outlook}

We have introduced a number of curvature measures and derived flows, as well as developing a family of coarse embedding kernels derived from some of them or from related operators. Furthermore, we have illustrated the ideas above on a sample of small and medium size networks, as well as on test images. However, these can represent only a preamble to an in-depth 
examination of these new tools. Therefore, first and foremost, there is the need for systematic experiments with large scale networks. It should be noted that this construction of networks, from manifolds to triangulations, and their subsequent simplification via curvature-based sampling is of high relevance in Deep Learning \cite{L+}, \cite{L++}.

A number of specific further directions of study impose themselves:
\begin{itemize}
	\item While, as we have discussed in detail, we are driven by the edge and higher order correlations centered approach, the node based one is still relevant. Therefore, the exploration of the capabilities of the (scalar curvature)  Yamabe flow is a natural and interesting venue of research.
	
	\item Of particular interest is the exploration of flows in the study of the long-time evolution of (hyper-)networks \cite{WSJ2}. To this end, flows derived from all the types of curvature herein should be examined. 
	
	\item Experiments with curvature-driven triangulations of manifolds and their sampling, especially for such manifolds as arising in Deep Learning (see references mentioned above) should definitely represent a future goal.

	\item One should try other embedding techniques than the one used herein, for instance the so called  {\it local linear embedding} \cite{RS}.
	
	\item Besides the kernel embedding method embraced in the present paper, one should explore another, and perhaps better established approach, that is the one using the eigenvectors and eigenfunctions of the Laplacian for embedding and sampling, where the classical graph Laplacian is replaced by the Bochner and rough ones. 
	Furthermore, given that they exist for each and every dimension up to the maximal one, they suggest themselves as a potentially powerful method of studying hypernetworks/complexes at many scales.
	
	\item The stronger results connecting Forman curvatures of discretizations, namely the convergence of the said curvatures to their manifold counterparts, as well as of the associated Laplacians to the smooth ones, is a goal that definitely should be pursued. 
	
\end{itemize}


\section*{Acknowledgement}

We would like to thank the anonymous reviewers, as well Frank Morgan for their valuable critique and constructive  suggestions which greatly helped improve our paper.


\section*{Appendix 1: Coarse Spaces}

We bring below the formal definition of a coarse space and give some basic examples, which in part are also relevant to our work.
To this end we first need a couple of preliminary definitions:

\begin{defn}
Let $X$ be a set and let $E, E_1, E_2 \subseteq \mathcal{P}(X \times X)$.  
\begin{enumerate}
	\item $E^{-1} = \{(x',x)\,|\, (x,x') \in E\}$ is called the {\it inverse} of $E$. 
	
	\item $E_1 \circ E_2 = \{(x_1,x_2)\,|\, (x_1,x) \in E_1\,, (x,x_2) \in E_2\}$ is called the {\it product} of $E_1$ and $E_2$.
\end{enumerate}
\end{defn}

\begin{defn}
	Let $X$ be a set and let $\mathcal{E} \subseteq \mathcal{P}(X \times X)$, such that the diagonal $\Delta(X) \subset \mathcal{E}$ and such that $\mathcal{E}$ is closed under the formation of subsets, inverses, products and (finite) unions. 
	$\mathcal{E}$ is called a {\it coarse structure} on $X$, and $(X,\mathcal{E})$ is called a coarse space. Furthermore, 
the sets $E \in \mathcal{E}$ are called the {\it controlled sets} or the {\it entourages} for the coarse structure. 
\end{defn}

\begin{exmps} We bring below a number of important examples. 
	\begin{enumerate}
		\item $\mathcal{E} = \mathcal{P}(X \times X)$ is called the {\em maximal} coarse structure on $X$. 
		
		\item Let $\mathcal{E}  \subseteq \mathcal{P}(X \times X)$ such that it contains only a finite number of points $(p,q) \notin \Delta(X)$. Then $\mathcal{E}$ is a coarse structure on $X$ called the {\em discrete coarse structure} on $X$. 
		
		
		
		\item Let $(X,d)$ be a metric space and let $\mathcal{E}  = \{E\,|\,E \subseteq \mathcal{P}(X \times X)\}$ such that $\sup\{d(x_1,x_2)\,|\, (x_1,x_2) \in E \}$ is finite. Then $\mathcal{E}$ is a coarse structure on $X$ called the {\em bounded coarse structure} on $X$. 
	\end{enumerate}
\end{exmps}

For further properties and characterizations of coarse spaces see \cite{Roe}.


\section*{Appendix 2: The General Formulas for the Forman-Laplacian and Curvature Functions}

We bring below the general formulas for the Forman-Laplacian and curvature functions. For more background on $CW$ complexes we refer the reader to Forman's paper \cite{Fo} as well as to \cite{Hat}.

Let $M$ be a (positively) weighted {\it quasiconvex} {\it regular} $CW$ complex, let $\alpha = \alpha^p \in M$ a $p$-dimensional cell and let $w(\alpha)$ denote its weight. While general weights are possible, making the combinatorial Ricci curvature extremely versatile, it is suffices (cf. \cite{Fo}), Theorem 2.5 and Theorem 3.9) to restrict oneself only to so called {\it standard weights}:

\begin{defn}  \label{saucan-def:st-wgh}
	The set of weights $\{w_\alpha\}$ is called a {\em standard set of
		weights} iff there exist $w_1, w_2 > 0$ such that given a $p$-cell
	$\alpha^p$, the following holds:
	\[w(\alpha^p) = w_1\cdot w_2^p\,.\]
\end{defn}
\noindent (Note that the combinatorial weights $w_\alpha \equiv 1$ represent a
set of standard weights, with $w_1 = w_2 = 1$.)

Using standard weights, we obtain the following formula for the Forman-Laplacian
\begin{align}\label{eq:bochner}
\Box_p(\alpha_1^p,\alpha_2^p) =
\sum_{\substack{\beta^{p+1}>\alpha_1^p \\ \beta^{p+1}>\alpha_2^p}}\epsilon_{\alpha_1,\alpha_2,\beta}\frac{\sqrt{\omega(\alpha_1^p)\omega(\alpha_2^p)}}{\omega(\beta^{p+1})} +
\sum_{\substack{\gamma^{p-1}<\alpha_1^p \\ \gamma^{p-1}<\alpha_2^p}}\epsilon_{\alpha_1,\alpha_2,\gamma}\frac{\omega(\gamma^{p-1})}{\sqrt{\omega(\alpha_1^p)\omega(\alpha_2^p)}}\,,
\end{align}
where, for instance $\alpha < \beta$ denotes that $\alpha$ is a face of $\beta$, and  $\epsilon_{\alpha_1,\alpha_2,\beta},\epsilon_{\alpha_1,\alpha_2,\gamma} \in \{-1,+1\}$ represent the respective {\it incidence numbers} of the cells $\beta$ relative to $\alpha_i$, and $\alpha_i$ relative to $\gamma$, $i = 1,2$, respectively. (see \cite{Hat}).

Furthermore, we also obtain the formula for the curvature functions, namely
\begin{align}\label{eq:Forman}
\begin{aligned}
\mathcal{F}(\alpha^p) &= \omega(\alpha^p)\Big[\Big(\sum_{\beta^{p+1}>\alpha^p}\frac{\omega(\alpha^p)}{\omega(\beta^{p+1})}\;
+ \sum_{\gamma^{p-1}<\alpha^p}\frac{\omega(\gamma^{p-1})}{\omega(\alpha^p)}\Big) \\
&\quad -\sum_{\alpha_1^p\parallel \alpha^p, \alpha_1^p \neq \alpha^p}\Big|\sum_{\substack{\beta^{p+1}>\alpha_1^p \\ \beta^{p+1}>\alpha^p}}\frac{\sqrt{\omega(\alpha^p)\omega(\alpha_1^p)}}{\omega(\beta^{p+1})}\:
- \sum_{\substack{\gamma^{p-1}<\alpha_1^p \\ \gamma^{p-1}<\alpha^p}}\frac{\omega(\gamma^{p-1})}{\sqrt{\omega(\alpha^p)\omega(\alpha_1^p)}}\Big|\:\;\Big]\,;
\end{aligned}
\end{align}
\noindent where $\alpha < \beta$ denotes $\alpha$ being a face of $\beta$ and $\alpha_1 \parallel \alpha_2$ parallel faces $\alpha_1$ and $\alpha_2$, the notion
of parallelism being defined as follows:
\begin{defn}  \label{saucan-def:parallel}
	Let $\alpha_1 = \alpha_1^p$ and $\alpha_2 = \alpha_2^p$ be two
	p-cells. $\alpha_1$ and $\alpha_2$ are said to be  {\em parallel}
	($\alpha_1
	\parallel \alpha_2$) iff either:
	(i) there exists $\beta = \beta^{p+1}$, such that $\alpha_1, \alpha_2 <
	\beta$; or (ii) there exists $\gamma = \beta^{p-1}$, such that
	$\alpha_1, \alpha_2 > \gamma$ holds, but not both simultaneously. (For example, in Fig. 1, $e_1,e_2,e_3,e_4$ are all the edges parallel to $e_0$.)
\end{defn}


\section*{Appendix 3: Standard Metrics on Weighted Graphs}

We present here briefly the two metric on graphs taking into account both node and edge weights that we mentioned in the main part of the paper. 

\begin{itemize}
	\item The {\it degree path metric} was fist introduced in \cite{DK} and further expanded upon in \cite{Huang-thesis}. It is closely associated with the (combinatorial) Laplacian, as well as with the random walk on a graph, and as such is widely used in theoretical studies of the Geometry of graphs. 

\begin{defn}
	Let $(X,m,w)$ be a node and edge weighted graph, $m$ denoting the node weights, and $w$ the edge weights. Then the function $\rho: X\times X \rightarrow [0,\infty)$,
       \[
		\rho(x,y) = \inf_\pi\sum_{i=1}^n\left(\max\{d(x_{i-1}),d(x_i)\}\right)^{-1/2}\,;
		\]
	where the infimum is taken all paths $\pi = x=x_0x_1 \ldots x_n=y$, and where $d$ denotes the {\it weighted degree}
	   \[
		d(x) = \frac{1}{m(x)}\sum_{y \sim x}w(x,y)\,,
		\]
	represents a metric on $X$.
\end{defn}

Since, one the one hand, the larger the degree of a node $x$, the faster the random walk leaves it and, on the other hand, by the definition of $\rho$, the larger the degree of any of the vertices $x,y$, the smaller the distance $\rho(x,y)$ between them, it follows that the faster the jump along an edge, the shorter the edge is in w.r.t. the metric $\rho$.
(For more details, such as the analogy with the Riemannian manifolds case, see \cite{Ke} and the references contained  therein.)

\item As its name suggests, the definition of the resistance metric is motivated by electrical networks, and $1/\Omega(x,y)$ represents an abstraction of the notion of {\it electrical conductance} and, as such, measures the connectivity along the edge $e$ connecting the nodes $x$ and $y$.  More precisely, we have the following definition: 

\begin{defn} Let $G$ be as above. The {\it resistance metric} $\Omega$ on $G$ is defined as 
	\begin{equation}
	\Omega(x,y) = \frac{1}{\sum_{t \in V}f(t)r(t,y)}\,; ;
	\end{equation}
\end{defn}
where 
\[
r = \left\{ \begin{array}{ll}
     \mbox{$\frac{1}{w(e)}$}\,, & \mbox{$x \sim y$}\,; \\
     1 \,,                     & \mbox{$x \not \sim y$} 
     \end{array}
     \right. 
\]
and where $f:V \rightarrow [0,1]$ is the unique function such that (a) $f(x) = 1, f(y) = 0$; and (b) $\sum_{z \in V}(f(t) - f(z))r(t,z) = 0$, for any $t \neq x,y$.

The resistance metric represents the weighted average of all the paths $\pi$ of ends $x$ and $y$, and it is, therefore, best used when the number of such paths is important. 

It is however, essentially a vertex-weights induced metric, except for the role of the function $f$ which is defined on the set of vertices. The following probabilistic interpretation however allows us to view the resistance metric as being induce both by vertex and edge weights: 
\[
\Omega(x,y) = \frac{1}{d(x){\rm Pr} (x \rightarrow y)}\,;
\]
where $d$ denotes here the degree of $x$ and ${\rm Pr} (x \rightarrow y)$ is the probability of a random walk starting from $x$ to reach $y$ before it returns to $x$. (For more details see \cite{DD} and the references therein.) If one considers, as it is commonly done, that the jump probability to any of the neighbors of a given vertex is equal, than this probability is viewed as a weight on the said vertex, thus allowing us to consider $\Omega$ as a metric on a node and edge weighted graph. 

\end{itemize}


\end{document}